\documentclass[12pt]{amsart}

\usepackage{cancel}
\usepackage{latexsym, amssymb, amsmath}
\usepackage{soul,esint}
\usepackage{amsfonts, graphicx}
\usepackage{graphicx,color}
\newcommand{\be}{\begin{equation}}
\newcommand{\ee}{\end{equation}}
\newcommand{\beq}{\begin{eqnarray}}
\newcommand{\eeq}{\end{eqnarray}}

\newtheorem{thm}{Theorem}[section]

\newtheorem{lma}{Lemma}[section]
\newtheorem{prop}{Proposition}[section]
\newtheorem{cor}{Corollary}[section]
\newtheorem{defn}{Definition}[section]
\theoremstyle{remark}
\newtheorem{rem}{Remark}[section]
\numberwithin{equation}{section}

\newtheorem{claim}{Claim}[section]

\newtheorem{subclaim}{Subclaim}[section]

\newcommand{\abs}[1]{\left|#1\right|}

\newcommand{\Ric}{\textup{Ric}}

\newcommand{\KE}{\textup{KE}}
\newcommand{\pt}{\mathbf{P}}

\newcommand{\R}{\mathbb{R}}
\newcommand{\C}{\mathbb{C}}

\newcommand{\K}{\mathcal{K}}
\newcommand{\V}{\V}
\newcommand{\CP}{\mathbb{CP}}

\newcommand{\e}{\varepsilon}

\renewcommand{\V}{\mathcal{V}}

\def\heat{\left(\frac{\partial}{\partial t}-\Delta \right)}
\def\a{{\alpha}}
\def\b{{\beta}}

\newcommand{\ddb}{\sqrt{-1}\partial\bar\partial}

\newcommand{\D}[2]{\frac{\partial #1}{\partial #2}}
\newcommand{\ddbar}{\partial \bar{\partial}}

\def\be{\begin{equation}}
\def\ee{\end{equation}}
\def\bee{\begin{equation*}}
\def\eee{\end{equation*}}

\def\lf{\left}
\def\ri{\right}

\def\K{K\"ahler }
\def\KR{K\"ahler-Ricci }

\def\Ric{\text{\rm Ric}}

\def\p{\partial}
\def\ddb{\sqrt{-1}\partial\bar\partial}

\def\heat{\lf(\frac{\p}{\p t}-\Delta\ri)}
\def\sheat{(\partial_t-\Delta)}

\def\e{\varepsilon}
\def\a{{\alpha}}
\def\b{{\beta}}

\def\R{\mathbb{R}}
\def\C{\mathbb{C}}
\begin{document}
\title[Higher-Order Estimates of Long-Time Solutions to KRF]{Higher-Order Estimates of Long-Time Solutions to the K\"ahler-Ricci Flow}
\author{Frederick Tsz-Ho Fong$^1$}
\address[Frederick Tsz-Ho Fong]{Department of Mathematics, Hong Kong University of Science and Technology}
 \email{frederick.fong@ust.hk}
\thanks{$^1$Research partially supported the Hong Kong RGC General Research Funds \#16300018 and \#16304719. }
 
 \author{Man-Chun Lee$^2$}
\address[Man-Chun Lee]{Department of Mathematics, Northwestern University, 2033 Sheridan Road, Evanston, IL 60208}
\email{mclee@math.northwestern.edu}
\thanks{$^2$Research partially supported by NSF grant DMS-1709894.}


\renewcommand{\subjclassname}{
  \textup{2010} Mathematics Subject Classification}

\date{\today}

\begin{abstract}
In this article, we study the higher-order regularity of the K\"ahler-Ricci flow on compact K\"ahler manifolds with semi-ample canonical line bundle. We proved, using a parabolic analogue of Hein-Tosatti's work on collapsing Calabi-Yau metrics, that when the generic fibers of the Iitaka fibration are biholomorphic to each other, the flow converges in $C_{\textup{\textup{loc}}}^\infty$-topology away from singular fibers to a negative K\"ahler-Einstein metric on the base manifold. In particular, we proved that the Ricci curvature of the flow is uniformly bounded on any compact subsets away from singular fibers when the generic fibers are biholomorphic to each other.
\end{abstract}


\maketitle

\markboth{}{Higher-order estimates of the \KR flow}

\section{introduction}
In this article, we study the normalized K\"ahler-Ricci flow:
\[
\D{g(t)}{t} = -\Ric(g(t))-g(t)
\]
on a compact K\"ahler manifold $X$ with semi-ample canonical line bundle $K_X$. Such a manifold admits a Iitaka fibration structure given by a holomorphic map $f : X \to \Sigma \subset \CP^N$ with possibly singular fibers and possibly singular base manifold $\Sigma$. Let $S \subset \Sigma$ be the union of the set of singular values of $f$ and the singular set of $\Sigma$. The regular (also known as generic) fibers $f^{-1}(z)$, where $z \in \Sigma \backslash S$, are Calabi-Yau manifolds. The dimension of $\Sigma$ is the Kodaira dimension of $X$. We focus on the case when $0 < \dim \Sigma < \dim X$, and we let $\dim_{\mathbb{C}} \Sigma = m$ and $\dim_{\mathbb{C}} X = m+n$, so that the Calabi-Yau fibers have complex dimension $n$.

The K\"ahler-Ricci flow under this setting (assuming $0 < \dim \Sigma < \dim X$) has been extensively studied by various authors \cite{ST07,ST12,ZZhang14,Gill,FZ15,ST16,TWY,TZ,TianZLZ18,
ZhangYSI,ZhangYSII,FongYSZ19, ZhangYS19,STZ19,Jian18,JianShi19,GTZ19}. In particular, the semi-ampleness of $K_X$ implies the nefness and hence the flow has a long-time solution by the results of \cite{Cao85,Tsu,TZ06}. When $X$ is projective, the abundance conjecture predicts that the nefness of $K_X$ is equivalent to its semi-ampleness. And hence, it is natural and tempting to consider the behaviour of the flow as $t\rightarrow \infty$ under the above setting. In \cite{ST07} and \cite{ST12}, Song-Tian proved that the flow will converge to a generalized K\"ahler-Einstein metric in the sense of measure on the base manifold $\Sigma$ as $t\rightarrow +\infty$. The generalized K\"ahler-Einstein metric $\omega_\Sigma$ satisfies $\Ric(\omega_\Sigma) = -\omega_\Sigma + \omega_{\text{WP}}$, where $\omega_{\textup{WP}}$ is the Weil-Petersson's term which reflects the variation of complex structures of the fibers, which vanishes when the fibers are biholomorphic to each other. They conjectured that the regularity of convergence can be improved to $C^\infty_{\textup{\textup{loc}}}(f^{-1}(\Sigma\backslash S))$-convergence, and also (global) Gromov-Hausdroff convergence. These conjectures are open in general, although many progress have been made by many authors toward them (in particular, \cite{Gill,FZ15,HT15,TWY,TianZLZ18,STZ19,GTZ19}).

Concerning the above-mentioned conjecture about the regularity, it was proved by Gill \cite{Gill} that when $X$ is a direct product $E \times \Sigma$ of a torus $E$ (or its finite quotient) and a compact \K manifold $\Sigma$ with ample canonical line bundle $K_\Sigma$, the normalized K\"ahler-Ricci flow converges in $C^\infty$-topology to the K\"ahler-Einstein metric on $\Sigma$. The first-named author and Z.Zhang proved in \cite{FZ15} the $C_{\textup{\textup{loc}}}^\infty(f^{-1}(\Sigma\backslash S))$-convergence of the flow when the regular fibers of $f : X \to \Sigma$ are complex tori (with possibly different complex structures) using a parabolic analogue of Gross-Tosatti-Y.G.Zhang's work \cite{GTZ13} on collapsing Calabi-Yau metrics. The rationality assumption of \cite{FZ15} was later removed by Hein-Tosatti in \cite{HT15}. For general Calabi-Yau fibrations $f : X \to \Sigma$, Tosatti-Weinkove-Yang proved in \cite{TWY} the $C^0_{\textup{\textup{loc}}}(f^{-1}(\Sigma\backslash S))$-convergence of the metric to the generalized K\"ahler-Einstein metric on the base manifold $\Sigma$.

The main goal of this article is to establish the $C^\infty_{\textup{\textup{loc}}}(f^{-1}(\Sigma\backslash S))$-convergence of the flow when the regular fibers are biholomorphic to each other.

\begin{thm}
\label{thm:Main}
Consider the K\"ahler-Ricci flow $g(t)$ under the above setting. Suppose there exists an open set $B$ contained compactly inside $\Sigma \backslash S$ (i.e. $B \subset\subset \Sigma \backslash S$) over which the fibers $\{f^{-1}(z)\}_{z \in B}$ are all biholomorphic to each other, so that $f^{-1}(B)$ can be trivialized as a product $B \times Y$, where $Y$ is a Calabi-Yau manifold. Denote $g_P(t) = g_{\C^m} + e^{-t}g_Y$ where $g_{\C^m}$ is the Euclidean metric on $B$, and $g_Y$ is a fixed Calabi-Yau metric on $Y$. Then, for any $\Omega\Subset B$ and $k \in \mathbb{N}$, there exists $C_{k,\Omega}$ such that
\[\sup_{(\Omega \times Y) \times [0,\infty)} \abs{\nabla^{k, g_P(t)}g(t)}_{g_P(t)} \leq C_{k,\Omega}.\]
In particular,if all fibers of $f$ over $\Sigma \backslash S$ are isomorphic, then the \KR flow $\omega(t)$ will converge to $f^*\omega_\infty$ in $C^\infty_{loc}(f^{-1}(\Sigma \backslash S))$ as $t\rightarrow +\infty$.
\end{thm}

The higher-order regularity result implies that the Ricci curvature stays bounded on $f^{-1}(B)$, as the reference metric $g_P(t)$ is Ricci-flat and $g(t)$ is uniformly equivalent to $g_P(t)$ (see \eqref{metric-equ}) on any compact subset of $f^{-1}(\Sigma \backslash S)$.
\begin{cor}
\label{cor:RicciBound}
If the Iitaka fibration is locally product on $f^{-1}(B)$, then $\Ric(t)$ stays bounded as $t\rightarrow +\infty$ on $f^{-1}(K)$ for any compact subset $K \subset\subset B$.
\end{cor}

The uniform Ricci curvature bound on compact subsets of $f^{-1}(\Sigma \backslash S)$ is the common hypothesis in many recent works about the K\"ahler-Ricci flow on compact K\"ahler manifolds with semi-ample $K_X$, such as Y.S.Zhang \cite{ZhangYS19}, Tian-Z.L.Zhang \cite{TianZLZ18}, Song-Tian-Z.L.Zhang \cite{STZ19} and Gross-Tosatti-Y.G.Zhang \cite{GTZ19}. These works address the uniform diameter bound and the Gromov-Hausdorff convergence of the K\"ahler-Ricci flow under various assumptions such as Kodaira dimensions, crossing of singular sets, etc. Now that we have established such a uniform Ricci curvature bound in the case of biholomorphic generic fibers, we can further improve some of the earlier works. This partially resolved Song-Tian's conjecture on the uniform diameter bound along the normalized \KR flow when the Iitaka fibration is locally product.

\begin{cor}\label{cor:diameter}
Let $X$ be a \K manifold with semi-ample $K_X$ such that the Iitaka fibration is locally product on $f^{-1}(B)$. Suppose $g(t)$ is the solution of the normalized K\"ahler-Ricci flow on X with any initial \K metric $g_0$, then there exists $D>0$ such that for all $t\in [0,+\infty)$,
$$\mathrm{diam}\left(X,g(t)\right)\leq D$$
where $\mathrm{diam}\left(X,g(t)\right)$ is the (global) diameter of $X$ measured by $g(t)$.
\end{cor}

When the generic fibers are complex tori, this was known by works of \cite{Gill, FZ15, HT15, TianZLZ18, STZ19}. The uniform diameter bound is essential for us to extract a convergent sub-sequence $(X,d_{g(t_i)})$ in Gromov-Hausdroff's topology. By combining Corollary \ref{cor:RicciBound} with results in \cite{STZ19,GTZ19}, one can identify the Gromov-Hausdorff limit in the cases described below:



\begin{cor}
\label{cor:GHConvergence}
Let $X$ be a K\"ahler manifold with semi-ample $K_X$ such that the Iitaka fibration is locally product. 
Suppose further that $\Sigma$ is smooth and the codimension $1$ irreducible components of the singular set $S$ have simple normal crossings (see \cite{GTZ19} for precise definitions),
then $(X, d_{g(t)})$ converges in Gromov-Hausdorff's topology to the metric completion of $(\Sigma \backslash S, d_{g_{\KE}})$ which is homeomorphic to $\Sigma$. In particular, it is always the case when $X$ has Kodaira dimension $1$ (c.f. \cite{ZhangYSII,TianZLZ18}).
\end{cor}
The proof of Theorem \ref{thm:Main} adapts the idea in Hein-Tosatti's work \cite{HT18} on the collapsing of Calabi-Yau metrics, which is an elliptic complex Monge-Amp\`ere problem. In order to apply the idea of \cite{HT18} on the K\"ahler-Ricci flow, we need to establish a parabolic version of cylindrical Schauder estimates. This is done in Section \ref{sect:Schauder} of this paper. Furthermore, unlike in \cite{HT18} whose metrics involved are always Ricci-flat, we have to use a parabolic rescaling and dilation (instead of only rescaling) when adapting the blow-up analysis in \cite{HT18}. Thanks to the result by Song-Tian \cite{ST16} which shows the scalar curvature is always uniformly and globally bounded, we are able to argue that the limit metric of the blow-up sequence is also Ricci-flat. The Liouville's Theorem used in \cite{HT18} can then be applied in our setting. This is done in Section \ref{sect:LocalEstimates} of this paper.

Acknowledgement: The authors are grateful to Valentino Tosatti and Hans-Joachim Hein for suggesting the problem and generous sharing of the ideas. This work cannot be done without their support. Part of the works was done when the second author visited the Institute of Mathematical Science at The Chinese University of Hong Kong, which he would like to thank for the hospitality.

\section{Preliminary}
First let us start with the classical gradient estimates on manifolds with non-negative Ricci curvature. This will be used to study the ancient solutions to heat equation.
\begin{prop}\label{gradient-esti}
Let $(M^n,g)$ be a complete non-compact Riemannian manifold with $\Ric(g)\geq 0$. Suppose $u$ is a solution to heat equation, $\sheat u=0$ on $Q_{p,R}=B_{g}(p,R)\times (-R^2,0]$ for some $R>0$ and $p\in M$, then
\begin{equation}
\sup_{Q_{p,R/2}}|\nabla u|\leq \frac{C_n}{R} \sup_{Q_{p,R}}|u|.
\end{equation}
\end{prop}
\begin{proof}
The proof is standard. For reader's convenience, we include the proof here. By considering $\tilde g=R^{-2}g$ and $\tilde u(x,t)=u(x,R^2t)$, we may assume $R=1$. By the heat equation of $u$, we have 
\begin{equation}
\left\{ 
\begin{array}{ll}
\sheat |\nabla u|^2&\leq -2|\nabla^2 u|^2;\\\\
\sheat |u|^2&= -2|\nabla u|^2.
\end{array}
\right.
\end{equation}
where we have used $\Ric\geq 0$ on the evolution equation of $|\nabla u|^2$. On the other hand, by Laplacian comparison there is $C_n>0$ such that for all $d_g(x,p)\geq \frac{1}{4}$, $\Delta d_g(x,p)\leq C_n$ in the barrier sense. By the trick of Calabi, we may assume $d_g(x,p)$ to be smooth when we apply maximum principle. Let $\Phi(x,t)=\phi(d_g(x,p))$ where $\phi$ is a smooth non-increasing function on $\mathbb{R}$ so that $\phi\equiv 1$ on $(-\infty,\frac{1}{2}]$, vanishes outside $(-\infty,1]$ and satisfies $ |\phi'|^2 \leq 100\phi$, $\phi''\geq -100$. Consider test function $F=(t+1)\Phi |\nabla u|^2+Au^2$ where $A$ is a constant to be fixed later. Then
\begin{equation}
\begin{split}
\heat F&\leq (t+1)|\nabla u|^2 \heat \Phi +\Phi |\nabla u|^2-2(t+1)\langle \nabla \Phi,\nabla |\nabla u|^2\rangle  \\
&\quad -2(t+1)\Phi |\nabla^2 u|^2-2A|\nabla u|^2\\
&\leq (C_n-2A) |\nabla u|^2\\
&<0
\end{split}
\end{equation}
provided that we choose $A>C_n$. Result follows from maximum principle.
\end{proof}

By letting $R\rightarrow +\infty$, the following Liouville theorem is immediate.
\begin{cor}\label{gradient-esti-LV}
Under the assumptions in Proposition \ref{gradient-esti}, if $u$ is a function on $M\times (-\infty,0]$ of $o(|t|^2+d_{g}(x,p))$, then $u$ is a constant function.
\end{cor}

Next, we need the local estimates of the \KR flow which is an consequence of Evans-Krylov theory \cite{Ev,Kv}, see also \cite{ShW12} for a proof using maximum principle. By parametrizing the time, we have the following local estimates of \KR flow. This will be used extensively in the rest of the paper. 

\begin{thm}\label{KRF-local-estimates}
Let $B_1(0)$ be a Euclidean unit ball and $g(t)$ is a solution to $\partial_t \omega=-\Ric(\omega)-k\omega$ on $B_1(0)\times [0,T]$ so that 
$$A^{-1} g_{\mathbb{C}^n}\leq g(t)\leq A g_{\mathbb{C}^n}$$
for some $A>1,|k|\leq k_0$. Then for all $m\in \mathbb{N}$, there exist $C(n,m,T,A,k_0)$'s such that for all $t\in [\frac{1}{2}T,T]$, 
$$\sup_{B_{1/2}(0)}|\nabla^{m,g_{\mathbb{C}^n}}g(t)|\leq C(n,m,T,A,k_0).$$
\end{thm}

We also need the following Liouville theorems for the Ricci flat \K metrics on $\mathbb{C}^n$ and $\mathbb{C}^n\times Y$ where $Y$ is a compact Calabi-Yau \K manifold.
\begin{thm}[\cite{RbSch84}]\label{Liouville-flat}
Let $\omega$ be a Ricci flat \K metric on $\mathbb{C}^n$ such that 
$$A^{-1}\omega_{\mathbb{C}^n}\leq \omega\leq A\omega_{\mathbb{C}^n}$$
for some $A>1$ where $\omega_{\mathbb{C}^n}$ is the standard Euclidean metric on $\mathbb{C}^n$, then $\omega$ is constant.
\end{thm}

The Liouville theorem on product background was proved by Hein in \cite{Hein19}, see also \cite{LLZ17} for an alternative proof using Li's mean value Theorem \cite{Li86}. Here we state a slightly simplified form which is cleaner and sufficient for our purpose.
\begin{thm}\label{Liouville-product}
Let $\omega$ be a Ricci flat \K metric on $\mathbb{C}^n\times Y$ where $Y$ is compact \K manifold with $\Ric_Y=0$. Suppose there exists $A>1$ such that 
$$A^{-1}(\omega_{\mathbb{C}^n}+\omega_Y)\leq \omega\leq A(\omega_{\mathbb{C}^n}+\omega_Y)$$
on $\mathbb{C}^n\times Y$ and $\omega$ is $d$-cohomologous to $\omega_{\mathbb{C}^n}+\omega_Y$, then $\omega$ is parallel with respect to $\omega_P=\omega_{\mathbb{C}^n}+\omega_Y$.
\end{thm}

\begin{rem}
In term of \KR flow, Theorem \ref{Liouville-flat} and \ref{Liouville-product} can be interpreted as a gap theorem of static solution. In fact, it is not difficult to generalize Theorem \ref{Liouville-flat} to the case when $\omega(t)$ is only an ancient solution to the \KR flow. It will be interesting to know if the same phenomenon is still true in the product case. More precisely, \textit{if $\omega(t)=\omega_P+\ddb \varphi(t)$ is an ancient solution to the \KR flow uniformly equivalent to $\omega_P$ for all $t<0$, then is it true that $\omega_P(t)\equiv \omega_P$?}
\end{rem}

\section{Schauder estimates on cylinders}
\label{sect:Schauder}
In this section, we will derive Schauder estimates for heat equations on cylinders. We start with the definition of H\"older semi-norm of a time depending tensor $\sigma(t)$.

\begin{defn}
Let $(M,g)$ be a complete Riemannian manifold and $E\rightarrow M$ be a vector bundle on $M$ with fiber metric $h$ and the $h$-preserving connection $\nabla$. If $x,y\in M$ and if there is a unique minimal $g$-geodesic $\gamma$ joining from $x$ to $y$, then we let $\pt^g_{x,y}$ be the $\nabla$-parallel translation on $E$ along $-\gamma$. Let $Q^g_{p,R}=B_g(p,R)\times (-R^2,0]$ be the parabolic cylinder where $B_g(p,R)$ is the $g$-geodesic ball of radius $R$ centred at $p\in M$. Then for all sections $\sigma\in C_{\textup{loc}}^{\a,\a/2}(Q_g(p,2R),E)$, we define the parabolic H\"older semi-norm to be
\begin{align*}
[\sigma]_{\a,\a/2,Q^g_{p,R},g}=\sup\left\{ \frac{|\sigma(x,t)-\pt^g_{x,y}(\sigma(y,s))|_h}{(d_g(x,y)+|s-t|^{1/2})^\a} \right\}
\end{align*}
where the sup is taken over all $(x,t),(y,s)\in Q^g_{p,R}$ such that $(x,t)\neq (y,s)$ and $P_{x,y}^g$ is defined.
\end{defn}

For notational convenience, we say that $\sigma\in C^{k+\a,1+\a/2}(U,g)$ if $\nabla^m \partial_t^l \sigma$ exist and are continuous for all $m+2l\leq k$, $l\leq 1$ on a open subset $U$ in space-time and $[\nabla^{k-2}\partial_t \sigma]_{\a,\a/2,U,g}+[\nabla^k \sigma]_{\a,\a/2,U,g}<+\infty$. We also define $||\sigma||_{\infty,Q^g_{p,R}}=\sup_{Q^g_{p,R}}|\sigma|_h$. In the following, we will drop the index $g$ of metric when the content is clear. The goal of this section is to prove the following parabolic Schauder estimates on cylinders. Throughout this section, we will assume 
\begin{center}
$\star$ $(Y,g_Y)$ is a closed \K manifold with $\Ric(g_Y)=0$.
\end{center}

\begin{thm}\label{Sch-esti}
Let $m\in \mathbb{N}$ and equip $\mathbb{C}^m\times Y$ with the product Riemannian metric $g_P=g_{\mathbb{C}^m}+ g_Y$. Then for any $k\in \mathbb{N}_{\geq 2},\,\a\in (0,1)$, there exists a constant $C(\a,m,k,g_Y)$ such that for all $x\in \mathbb{C}^m\times Y$ and $0<\rho<R$, 
\begin{align*}
[\nabla^{k,g_P} \eta]_{\a,\a/2,Q_{x,\rho}}
&\leq C\Big([\nabla^{k-2,g_P} \Box \eta]_{\a,\a/2,Q_{x,R}} +(R-\rho)^{-k-\a}||\eta||_{\infty,Q_{x,R}}\\
&\quad +(R-\rho)^{-k-\a+2}||\Box \eta||_{\infty,Q_R}\Big)
\end{align*}
for all $\ddb$-exact $2$ forms $\eta\in  C^{k+\a,1+\a/2}(Q_{x,2R})$, where $\Box=\partial_t-\Delta$ and $\Delta$ denotes the Hodge Laplacian with respect to $g_P$.
\end{thm}

First, we need some preparation lemmas. The following lemma provides a interpolation inequality on Riemannian cylinders.
\begin{lma}\label{interpolation}
 Let $E\rightarrow Y$ be a metric vector bundle with metric connection $\nabla$. Extend $E$ trivially to $\mathbb{C}^m\times Y$ and extend $\nabla $ by trivially adding $\nabla^{\mathbb{R}^m}$. Let $g_P$ be the product metric, then for all $k\in \mathbb{N}_{\geq 1}$, $\a\in (0,1)$, there is $C(m,g_Y,k,\a)>0$ such that 
$$\sum_{j=1}^k (R-\rho)^j ||\nabla^j \sigma||_{\infty,Q_{x,\rho}}\leq C(m,g_Y,k,\a) \left[(R-\rho)^{k+\a}[\nabla^k \sigma]_{\a,\a/2,Q_{x,R}}+||\sigma||_{\infty,Q_{x,R}} \right]$$
for all $x\in \mathbb{C}^m\times Y$, $0<\rho<R$ and $\sigma\in C^{k+\a,1+\a}(Q_{x,2R},E)$.
\end{lma}
\begin{proof}
By \cite[Lemma 3.5]{HT18}, there exists $C(m,g_Y,k,\a)$ such that for each fixed $t\in [-\rho^2,0]$,
$$\sum_{j=1}^k (R-\rho)^j ||\nabla^j \sigma(t)||_{\infty,B_{x,\rho}}\leq C\left((R-\rho)^{k+\a}[\nabla^k \sigma(t)]_{\a,B_{x,R}}+||\sigma(t)||_{\infty,B_{x,R}} \right)$$
where  $[\eta]_{\a,B_{x,R}}$ is given by 
$$\sup\left\{ \frac{|\eta(x)-\pt^g_{p,q}(\eta(y))|_h}{(d_g(x,y))^\a} : p,q\in B_{x,R}, p\neq q, \pt_{p,q}^g \;\text{is defined} \right\}.$$
Clearly, $[\eta(t)]_{\a,B_{x,R}}\leq [\eta]_{\a,\a/2,Q_{x,R}}$ for each $t\in(-\rho^2,0]$. Result follows by taking sup over $t\in [-\rho^2,0]$
\end{proof}

For later purpose, we need an interpolation between H\"older semi-norms. 
\begin{lma}\label{int-a}
Under the assumptions in Lemma \ref{interpolation}, for all $k\in \mathbb{N}$, $\a\in (0,1)$, there is $C(m,g_Y,\a,k)>0$ such that
\begin{align*}
(R-\rho)^{k+\a}[\nabla^{k,g_P} \sigma]_{\a,\a/2,Q_{x,\rho}}&\leq C\Big[(R-\rho)^{k+1+\a}[\nabla^{k+1,g_P}\sigma]_{\a,\a/2,Q_R}+||\sigma||_{\infty,Q_R}\\
&\quad +(R-\rho)^{k+1}||\nabla^{k-1,g_P}\partial_t \sigma||_{\infty,Q_R}+(R-\rho)^2||\partial_t\sigma||_{\infty,Q_R}\Big]
\end{align*}
for all $x\in \mathbb{C}^m\times Y$, $0<\rho<R$ and $\sigma\in C^{k+1+\a,1+\a/2}(Q_{x,2R},E)$
\end{lma}
\begin{proof}
Let $(p,t),(q,s)\in Q_{x,\rho}$ be two points such that $\pt^{g_P}_{p,q}$ is well-defined and 
$$[\nabla^{k,g_P}\sigma]_{\a,\a/2,Q_{x,\rho}}=\frac{|\nabla^{k,g_P} \sigma(p,t)-(\pt^{g_P}_{p,q}\nabla^{k,g_P} \sigma (q,s)) |_{g_P}}{(d(p,q)+|t-s|^{1/2})^\a}.$$

If $d(p,q)+|t-s|^{1/2}\geq R-\rho$, then the conclusion follows immediately from sup norm $||\nabla^{k,g_P}\sigma||_{\infty,Q_{x,R}}$ and Lemma \ref{interpolation}. Hence it suffices to consider $d(p,q)+|t-s|^{1/2}< R-\rho$. Recall
\begin{equation}
\begin{split}
&\quad \frac{|\nabla^{k,g_P} \sigma(p,t)-(\pt_{p,q}^{g_P}\nabla^{k,g_P}  \sigma)(q,s) |_{g_P}}{(d(p,q)+|t-s|^{1/2})^\a}\\
&\leq \frac{|\nabla^{k,g_P}  \sigma(p,t)-(\pt_{p,q}^{g_P}\nabla^{k,g_P}  \sigma)(q,t) |_{g_P}}{(d(p,q)+|t-s|^{1/2})^\a}+\frac{| \nabla^{k,g_P} \sigma(q,t)-\nabla^{k,g_P}  \sigma(q,s) |_{g_P}}{(d(p,q)+|t-s|^{1/2})^\a}\\
&=\mathrm{I}+\mathrm{II}.
\end{split}
\end{equation}

First we note that for any tensor $\tau$ and geodesic $\gamma$ emerging from $x=\gamma(0)$, 
\begin{equation}\label{MVT}
\begin{split}
\pt^{g_P}_{\gamma(0),\gamma(t_0)}\tau(\gamma(t_0))-\tau(x)&=\pt^{g_P}_{\gamma(0),\gamma(t)}\tau(\gamma(t))|^{t_0}_0\\
&=\int^{t_0}_0\frac{d}{dt}\left(\pt^{g_P}_{\gamma(0),\gamma(t)}\tau(\gamma(t)) \right)dt\\
&=\int^{t_0}_0 \pt^{g_P}_{\gamma(0),\gamma(t)}\nabla_{\dot\gamma} \tau (\gamma(t))\, dt.
\end{split}
\end{equation}

Hence, $(R-\rho)^{k+\a}\mathrm{I}$ can be controlled easily by the right hand side in the conclusions using Lemma \ref{interpolation}. It suffices to consider the second term $\mathrm{II}$. \\

{\bf Case 1. $|t-s|^{1/2}<\text{inj}(Y)$:} Fix any unit tangential direction $v$ and let $\gamma(t)=\exp_q^{g_P}(tv)$, $z=\gamma(\e)$ where $\e=|t-s|^{1/2}$. 
\begin{equation}
\begin{split}
&\quad |\nabla^{g_P}_v\nabla^{k-1,g_P}  \sigma(q,t)-\nabla^{g_P}_v\nabla^{k-1,g_P} \sigma(q,s)|\\
&\leq 
\left| \nabla^{g_P}_v\nabla^{k-1,g_P} \sigma(q,t)-\frac{1}{\e}\left(\pt^{g_P}_{q,\gamma(\e)}\nabla^{k-1,g_P}\sigma(z,t)-\nabla^{k-1,g_P}\sigma(q,t) \right)\right|\\
&\quad +\left| \nabla^{g_P}_v\nabla^{k-1,g_P} \sigma(q,s)-\frac{1}{\e}\left(P_{\gamma(\e)}\nabla^{k-1,g_P}\sigma(z,s)-\nabla^{k-1,g_P}\sigma(q,s) \right) \right|\\
&\quad +\frac{1}{\e} \Big|\left(\pt^{g_P}_{q,\gamma(\e)}\nabla^{k-1,g_P}\sigma(z,s)-\nabla^{k-1,g_P}\sigma(q,s) \right)\\
&\quad \quad \quad -\left(\pt^{g_P}_{q,\gamma(\e)}\nabla^{k-1,g_P}\sigma(z,t)
-\nabla^{k-1,g_P}\sigma(q,t) \right) \Big|\\
&=\mathbf{A}+\mathbf{B}+\mathbf{C}
\end{split}
\end{equation}
where $\pt^{g_P}_{y,\gamma(\e)} \tau(z)$ denotes the parallel transport of $\tau(z)$ along the geodesic from $z=\gamma(\e)$ to $y=\gamma(0)$.
Apply \eqref{MVT} repeatedly to $\mathbf{A}$, $\mathbf{B}$ and $\mathbf{C}$ yielding 
$$\mathbf{A}+\mathbf{B}+\mathbf{C}\leq 2\e||\nabla^{k+1,g_P}\sigma||_{\infty,Q_{x,R}}+\e ||\partial_t \nabla^{k-1,g_P}\sigma||_{\infty,Q_{x,R}}.$$
Hence, 
\begin{equation}
(R-\rho)^{k+\a}\mathrm{II}\leq C(R-\rho)^{k+1} \left(||\nabla^{k+1,g_P}\sigma||_{\infty,Q_{x,R}}+ ||\nabla^{k-1,g_P}\partial_t \sigma||_{\infty,Q_{x,R}}\right)
\end{equation}
where we have used $|t-s|^{1/2}+d(p,q)\leq R-\rho$. The conclusion follows from Lemma \ref{interpolation} since it is true for any unit direction $v$. \hfill
\\

 {
{\bf Case 2. $|s-t|^{1/2}\geq \text{inj}(Y)$:}
In this case, $R-\rho>d(p,q)+|s-t|^{1/2}\geq \text{inj}(Y)$. If $v$ is a unit vector in the base direction, then the argument in {\bf Case 1} can be carried over as we can choose $\gamma(t)$ to be a horizontal line.

 If $v$ is a unit vector in the fibre direction, then \cite[Lemma 3.3]{HT18} with appropriate choice of vector bundle $E$ (for example, see \cite[page 30]{HT18}) will
imply for any $(p,t)\in Q_{x,\rho}$,
$$|\nabla_{\bf f}\nabla^{k-1,g_P}\sigma(p,t)|\leq C[\nabla^{k+1,g_P}\sigma]_{\a,\a/2,Q_{x,\rho }}.$$
Here we use $\bf f$ to denote covariant derivatives in the fibre direction. Hence, 
\begin{equation}
\begin{split}
(R-\rho)^{k+\a}\mathrm{II}&= (R-\rho)^{k+\a} \frac{|\nabla_v\nabla^{k-1} \sigma(q,t)-\nabla_v\nabla^{k-1} \sigma(q,s)|_{g_P}}{(d(p,q)+|s-t|^{1/2})^\a}\\
&\leq C(R-\rho)^{k+\a}[\nabla^{k+1,g_P}\sigma]_{\a,\a/2,Q_{x,\rho}}\\
&\leq  C'(R-\rho)^{k+1+\a}[\nabla^{k+1,g_P}\sigma]_{\a,\a/2,Q_{x,\rho}}.
\end{split}
\end{equation}
This completes the proof by combining two cases.
 }

\end{proof}

We will denote $\heat =\Box$. We will first establish the main step of the proof of Theorem \ref{Sch-esti}.
\begin{prop}\label{sch-prep}
Under the assumption in Theorem \ref{Sch-esti}, for all $\e>0$, there exist $\delta_0, C>0$ such that for all $x\in \mathbb{C}^m\times Y$, $R>0$, $\ddb$-exact $2$ forms $\eta\in C^{2+\a,1+\a/2}(Q_{x,2R})$, $0<\delta\leq \delta_0$, we have 
\begin{align*}
 [\nabla^{k,g_P}\eta]_{\a,\a/2,Q_{x,\delta R}}\leq &\, \epsilon[\nabla^{k,g_P}\eta]_{\a,\a/2,Q_{x, R}}+C[\nabla^{k-2,g_P}\Box \eta]_{\a,\a/2,Q_{x,R}}\\
&+C\sum_{m\leq k} R^{-k-\a+m} ||\nabla^{m,g_P} \eta||_{\infty,Q_{x,\delta  R}}\\
&+C\sum_{m\leq k-2}R^{-k-\a+m+2} ||\nabla^{m,g_P}\Box \eta||_{\infty,Q_{x,\delta R}}.
\end{align*}
\end{prop}
\begin{proof}
We will modify the argument in  \cite[Proposition 3.9 \& Theorem 3.13]{HT18}. In the proof, all convergence means sub-sequence convergence for notational convenience. Suppose the statement is false, then there exists $\e>0$, and sequences of $x_i\in \mathbb{C}^m\times Y$, $R_i>0$, $\delta_i\leq i^{-1}$ and $\ddb$-exact 2 form $\eta_i$ such that 
\begin{align*}
[\nabla^{k,g_P}\eta_i]_{\a,\a/2,Q_{x_i,\delta_i R_i}}>\,& \e[\nabla^{k,g_P}\eta_i]_{\a,\a/2,Q_{x_i, R_i}}+i[\nabla^{k-2,g_P}\Box \eta_i]_{\a,\a/2,Q_{x_i,R_i}}\\
&+\sum_{m\leq k} iR_i^{-k-\a+m} ||\nabla^{m,g_P} \eta_i||_{\infty,Q_{x_i,\delta_i R_i}}\\
&+\sum_{m\leq k-2}i R_i^{-k-\a+m+2} ||\nabla^{m,g_P}\Box \eta_i||_{\infty,Q_{x_i,\delta_i R_i}}.
\end{align*}

Let $p_i=(z_i,y_i,t_i), q_i=(\bar z_i,\bar y_i, s_i)\in \mathbb{C}^m\times Y\times (-\infty,0]$ be two points in $Q_{x_i,\delta_iR_i}$ such that the H\"older seminorm on the left hand side is attained at $p_i$ and $\bar p_i$. 
We may also assume $[\nabla^{k,g_P}\eta_i]_{\a,\a/2,Q_{x_i,\delta_i R_i}}=1$ by rescaling.

Let $\mathcal{L}=\{\sigma: \nabla^{k+1,g_P}\sigma=\nabla^{k,g_P}\partial_t \sigma=\nabla^{k-1,g_P}\Box\sigma=\nabla^{k-2,g_P}\partial_t\Box \sigma=0 \}$. Define its $k$-jet at $p_i$ by 
$$J_i \eta=\left(\nabla^{l,g_P}\eta(p_i),\nabla^{m,g_P}\Box \eta(p_i)\right)_{0\leq l\leq k;\, 0\leq m\leq k-2}.$$
Define its partial $k$-jet $\mathcal{L}J_i f$ to be the $g_P(x_i)$-orthogonal projection of $J_i f$ onto the space $J_i(\mathcal{L})$. As $J_i$ is injective on $\mathcal{L}$, there exists a unique $\eta_i^\#\in \mathcal{L}$ with $J_i \eta_i^\#=\mathcal{L}J_i\eta_i^\#=\mathcal{L}J_i \eta_i$. And then we define $\tilde \eta_i=\eta_i -\eta_i^\#$.
\begin{claim}
There exists a constant $C$ such that after passing to a subsequence, 
\begin{equation}
\left\{
\begin{array}{ll}
1&=[\nabla^{k,g_P}\tilde \eta_i]_{\a,\a/2,Q_{x_i,\delta_iR_i}} >\e [\nabla^{k,g_P} \tilde \eta_i]_{\a,\a/2,Q_{x_i,R_i}}+i[\nabla^{k-2,g_P}\Box \tilde \eta_i]_{\a,\a/2,Q_{x_i,R_i}},\\\\
&\displaystyle \sum_{m\leq k} |\nabla^{m,g_P} \tilde \eta_i|(p_i)+\sum_{m\leq k-2} |\nabla^{m,g_P}\Box \tilde \eta_i|(p_i)\leq C,\\\\
&\displaystyle [\nabla^{k,g_P}\tilde \eta_i]_{\a,\a/2,Q_{x_i,\delta_iR_i}}=\frac{|\nabla^{k,g_P}\tilde \eta_i(p_i)-\pt^{g_P}_{(z_i,y_i),(\tilde z_i,\tilde y_i)}\nabla^{k,g_P}\tilde \eta_i(q_i)|}{\left[|t_i- s_i|^{1/2}+d\left((z_i,y_i),(\bar z_i,\bar y_i) \right)\right]^\a}
\end{array}
\right.
\end{equation}
\end{claim}
\begin{proof}[Proof of claim.]
The first inequality and the third equality follow directly from the definition of $\eta_i^\#$ since $$\nabla^{k+1,g_P}\eta_i^\#=\nabla^{k,g_P}\partial_t \eta_i^\#=\nabla^{k-1,g_P} \Box \eta_i^\#=\nabla^{k-2,g_P}\partial_t\Box\eta_i^\#=0$$ and hence 
$$[\nabla^{k-2,g_P}\Box  \eta_i^\#]_{\a,\a/2,Q_{x_i,R_i}}=[\nabla^{k,g_P}\eta_i^\#]_{\a,\a/2,Q_{x_i,R_i}}=0.$$

Moreover, the second inequality holds as long as $R_i$ is bounded since the projection map is norm decreasing. Therefore, it suffices to consider the case: $R_i\rightarrow +\infty$. Suppose the estimate does not hold,  
$$\mu_i=\sum_{m\leq k} |\nabla^{m,g_P} \tilde \eta_i|(p_i)+\sum_{m\leq k-2} |\nabla^{m,g_P}\Box \tilde \eta_i|(p_i)\rightarrow +\infty.$$
Moreover, since 
$$Q_{x_i,R_i}\supset Q_{(0,y_i),(1-\delta_i)R_i},$$
we have semi-H\"older estimates of $\tilde \eta_i$ and $\Box \tilde \eta_i$ on shifted cylindrical domain. Hence, the rescaled function $\mu_i^{-1}\tilde \eta_i$ has bounded H\"older norm and hence converges to $\eta_\infty$ which satisfies 
\begin{itemize}
\item $\sum_{m\leq k} |\nabla^{m,g_P} \eta_\infty|(p_\infty)+\sum_{m\leq k-2} |\nabla^{m,g_P}\Box  \eta_\infty|(p_\infty)=1$;
\item $\nabla^{k+1,g_P} \eta_\infty=\nabla^{k,g_P}\partial_t  \eta_\infty=\nabla^{k-1,g_P}\Box \eta_\infty=\nabla^{k-2,g_P}\partial_t \Box \eta_\infty=0$ from the vanishing of semi-H\"older norm.
\end{itemize}

By the proof of \cite[Proposition 3.11]{HT18}, $\eta_\infty$ is still $\ddb$-exact 2 form, it follows that $\eta_\infty \in \mathcal{L}$ and that $\mathcal{L}J_\infty \eta_\infty \neq 0$. However $\mathcal{L}J_i \tilde \eta_i=0$ from construction and hence yields a contradiction after passing to limit.
\end{proof}

Let $\lambda_i^{-1}=d_{g_P}\left((z_i,y_i),(\bar z_i,\bar y_i)\right)+|t_i- s_i|^{1/2}\leq 2\delta_i R_i$,  $g_i=\lambda_i^2 g_P$ and consider the rescaled function $u_i(t)=\lambda_i ^{\a+k}\tilde \eta_i(\lambda_i^{-2}t)$. Also denote $\tilde p_i=(z_i,y_i,\tilde t_i)$ and $\tilde q_i=(\bar z_i,\bar y_i,\tilde s_i)$ where $\tilde t_i=\lambda_i^2 t_i$ and $\tilde s_i=\lambda_i^2 s_i$. Hence, $d_{g_i}\left((z_i,y_i),(\bar z_i,\bar y_i)\right)+|\tilde t_i-\tilde s_i|^{1/2}=1$ and 
\begin{align}
\label{1}1=&[\nabla^{k,g_i}u_i]_{\a,\a/2, \tilde Q_{x_i,\lambda_i\delta_i R_i}}>\e [\nabla^{k,g_i} u_i]_{\a,\a/2,\tilde Q_{x_i,\lambda_i R_i}}+i [\nabla^{k-2,g_i}\Box u_i]_{\a,\a/2,\tilde Q_{x_i,\lambda_i R_i}}\\
\label{2}&\sum_{m\leq k} \lambda_i^{-\a-k+m}|\nabla^{m,g_i}u_i |(\tilde p_i)+\sum_{m\leq k-2} \lambda_i^{-\a-k+m+2}|\nabla^{m,g_i} \Box u_i|(\tilde p_i) \leq C
\end{align}

where we denote $\tilde Q_{p,r}=B_{g_i}(p,r)\times (-r^2,0]$. Moreover, H\"older seminorm (which is $1$ now) on the left hand side is still attained by $\tilde p_i$ and $\tilde q_i$. By translation, we will also assume $\tilde t_i=0$ and $\tilde s_i<0$. Noted also that for $i$ sufficiently large, we have
\begin{equation}\label{pointed-domain}
\tilde Q_{x_i,\lambda_i R_i}\supset \tilde Q_{(z_i,y_i),(1-\delta_i)\lambda_iR_i}\supset \tilde Q_{(z_i,y_i),i/3}.
\end{equation}
\\

{\bf Case 1. $\lambda_i\rightarrow +\infty$.}
In this case, $(\mathbb{C}^m\times Y, g_i,x_i)$ converges to $(\mathbb{C}^{m+n},g_{\mathbb{C}^{m+n}},0)$ in the Cheeger-Gromov sense. In particular, \eqref{2} will degenerate as $i\rightarrow +\infty$. We need to modify $u_i$ further in order to apply compactness. 
We may assume that $\nabla^{m,g_i}u_i(\tilde p_i)=\nabla^{l,g_i}\partial_t u_i(\tilde p_i)=0$ for all $m=0,...,k$ and $l=0,...,k-2$. This can be done by subtracting its $k$-th Taylor polynomial at $\tilde p_i$ under the normal coordinate around it (together with the time variables), see \cite[page 19]{HT18} for the detailed argument with $B(\tilde x_i,R)$ replaced by the parabolic domain $\tilde Q_{(z_i,y_i),R}$

We may $\tilde p_i\rightarrow p_\infty=(0,0,0)\in Q$. Therefore, \eqref{pointed-domain} and the H\"older semi-norm in \eqref{1} allow us to take $u_i\rightarrow u_\infty$ on any compact subset of $Q=\mathbb{C}^{m+n}\times (-\infty,0]$ at least in $C^{2+\a,1+\a}_{\textup{loc}}(Q)$. Denotes the Euclidean derivative by $D$. Considering the coefficients of $D^i \Box u$, we have $[D^{k-2} \Box u_\infty]_{\a,\a/2,Q}=0$ which in particular implies $\Box D^{k-2}u_\infty=D^{k-2}\Box u_\infty =0$ on $\mathbb{C}^{m+n}\times (-\infty,0]$ since $D^i\Box u_\infty (0)=0$ for all $i\leq k-2$. Moreover from the $\a$-H\"older semi-norm, we also have $|D^k u_\infty|=O(r^\a)$ where $r=|x|+|t|^{1/2}$. By applying Corollary \ref{gradient-esti-LV} on $D^k u_\infty$, we have $D^k u_\infty \equiv 0$ which contradicts with the left hand side on \eqref{1}. \\

{\bf Case 2. $\lambda_i\rightarrow \lambda_\infty \neq0$.}
Without loss of generality, we may assume $\lambda_\infty=1$. Hence, $(\mathbb{C}^m\times Y,g_i,x_i)$ converges to $(\mathbb{C}^m\times Y,g_P,(0,y_\infty))$ in classical $C^\infty_{\textup{loc}}$ sense. By \eqref{1}, \eqref{2} and translation on $\tilde t_i$, we may assume $u_i\rightarrow u$ which is defined on $Q=(\mathbb{C}^m\times Y) \times (-\infty,0]$ with suitable regularity inherited from \eqref{2}. Clearly, we have $[\nabla^{k,g_P} u]_{\a,\a/2,Q}\neq 0$ and $[\nabla^{k-2,g_P}\Box_{g_P} u]_{\a,\a/2,Q}=0$. By the argument in \cite[Proposition 3.11]{HT18}, we may assume $u=\ddb h$ for some function $h$. By \K identity, $[\nabla^{k-2,g_P}\ddb (\Box h)]_{\a,\a/2,Q}=0$ and hence $\Delta_{g_P} \Box_{g_P} h=\Box_{g_P}\Delta_{g_P}h=P$ for some polynomial $P(x)$ in $\mathbb{C}^m$ of degree $k-2$. Denote its average over each fiber by $\bar h=\fint_Y h \omega_Y^n$. Clearly, we have $\Delta_{g_P} \bar h=\overline {\Delta_{g_P}h}$ and $\Box_{g_P} (\Delta_{g_P} \bar h- {\Delta_{g_P} h})=0$. Denote $w=h-\bar h$.

\begin{claim}
We have $w\equiv 0 $ on $Q$.
\end{claim}
\begin{proof}[Proof of Claim.]
We first prove that $\Delta_{g_P} w\equiv 0$ on $Q$. Let $\hat w=\Delta_{g_P} w$ and $\phi_R(z,t)=\exp{\left(-\frac{|x|}{R}+\frac{t}{R}\right)}$. Consider the energy 
$$E_R(t)=\int_{\mathbb{C}^m\times Y}\hat  w^2 \phi_R \omega_P^{m+n}.$$
From $[\nabla^{k,g_P}\ddb h]_{\a,\a/2,Q}<+\infty$, $E_R$ is finite since $|\partial \hat w|+|\hat w|$ is of polynomial growth uniformly in $t$. Furthermore we can do integration by part due to the growth rate. For $R>>1$,
\begin{equation}
\begin{split}
\frac{d }{dt}E_R(t)&=\int_{\mathbb{C}^m\times Y } 2\hat w\phi_R \partial_t \hat w +\hat w^2\phi'_R\\
&=\int_{\mathbb{C}^m\times Y } 2\hat w\phi_R \Delta \hat w +\frac{1}{R}\hat w^2\phi_R\\
&\leq \int_{\mathbb{C}^m\times Y } -2|\nabla \hat w|^2 \phi_R +2|\hat w||\nabla\hat  w||\nabla \phi_R|+\frac{1}{R}\hat w^2 \phi_R\\
&\leq \int_{\mathbb{C}^m\times Y } -|\nabla \hat w|^2\phi_R +\frac{C_n}{R}\hat w^2\phi_R.
\end{split}
\end{equation}

Since $\int_Y \hat w \omega_Y^n=0$, we have $\int_Y \hat w^2 \leq C_Y\int_Y |\nabla \hat w|^2$ by Poincar\'e inequality on $Y$. Hence,
\begin{equation}
\begin{split}
\int_{\mathbb{C}^m\times Y} \hat w^2 \phi_R
&=\int_{\mathbb{C}^m } \left(\int_{\{x\}\times Y} \hat w^2d\mu_Y\right)  \phi_R d\mu_x\\
&\leq C_Y\int_{\mathbb{C}^m} \phi_R\left(\int_{{\{x\}\times Y}} |\nabla \hat w|^2d\mu_Y\right) d\mu_x\\
&\leq C_Y \int_{\mathbb{C}^m\times Y} |\nabla \hat w|^2\phi_R .
\end{split}
\end{equation}

Combines with the inequality of $E_R'(t)$, we conclude that for $R$ sufficiently large, 
\begin{equation}
\begin{split}
E_R'(t)\leq 0.
\end{split}
\end{equation}
Thus, for $t>s>-\infty$, $E_R(t)\leq E_R(s)$. Hence $\Delta w=0$ by letting $s\rightarrow -\infty$.

From $[\nabla^{k-2,g_P}\Box \ddb h]_{\a,\a/2,Q}=0$, $|\ddb w|$ is also of polynomial growth. Apply Moser iteration on each fiber $\{z\}\times Y$, $|w|$ is also of polynomial growth on $Q$. Since $g_P$ is Ricci flat and $w$ is harmonic, $|\partial w|$ is also of polynomial growth by Proposition \ref{gradient-esti} or Cheng-Yau's gradient estimate. Applying the above argument on time independent energy 
$$\tilde E_R=\int_{\mathbb{C}^m\times Y}w^2 \exp\left(-\frac{|x|}{R} \right) \omega_P^{m+n},$$
we can show that $\tilde E_R\equiv 0$ for sufficiently large $R$ and hence $w\equiv 0$.
\end{proof}

Since $g_P$ is a product metric, we can now regard $h$ as a function on $\mathbb{C}^m\times (-\infty,0]$ and satisfies $\Box_{\mathbb{C}^m} D^{k}\ddb h=0 $. By Proposition \ref{gradient-esti} on its coefficients, we have $D^{k} \ddb h\equiv c_k$ which contradicts with the non-vanishing semi-H\"older norm of $\nabla^{k,g_P}u$.\\

{\bf Case 3. $\lambda_i\rightarrow 0$.} In this case, $(\mathbb{C}^m\times Y,g_i,x_i)$ converges to $(\mathbb{C}^m, g_{\mathbb{C}^m},0)$ in the Gromov-Hausdorff sense. Let $F_i(x,y)=(\lambda_i^{-1}x,y)$ be a diffemorphism on $\mathbb{C}^m\times Y$. Replace $g_i$, $u_i$, $x_i$, $\tilde p_i$, $\tilde q_i$, $(z_i,y_i)$ and $(\bar z_i,\bar y_i)$ by their pull-back under $F_i$. Then $g_i=g_{\mathbb{C}^m}+\lambda_i^2 g_Y\rightarrow g_{\mathbb{R}^d}$ smoothly as a tensor on $\mathbb{C}^m\times Y$. By translation and compactness of $Y$, we may assume $\tilde p_i=(0,y_i,0)\rightarrow p_\infty\in Q$ and $\tilde q_i\rightarrow q_\infty\in Q$ where $Q=\mathbb{C}^m\times Y \times (-\infty,0]$. Using $g_i\leq g_P$ and \eqref{pointed-domain}, we still have 
\begin{align}
\label{new-1}
[\nabla^{k,g_P}u_i]_{\a,\a/2,Q_{(x_i,y_i),i/4}}&\leq C,\\
\label{new-2} 
\sum_{m\leq k} \lambda_i^{-\a-k+m}|\nabla^{m,g_P}u_i |(\tilde p_i)+\sum_{m\leq k-2}& \lambda_i^{-\a-k+m+2}|\nabla^{m,g_P} \Box u_i|(\tilde p_i) \leq C
\end{align}
for sufficiently large $i$. Here all norms and semi-norms are calculated with respect to the product metric $g_P$. Hence, we may let $u_i\rightarrow u$ in at least $C^{2,1}_{\textup{loc}}$ and satisfies 
\begin{align}
\label{new-3}\nabla^{m,g_P}u(p_\infty)=\nabla^{l,g_P}\partial_t u (p_\infty)=0
\end{align}
for all $m\leq k$ and $l\leq k-2$. Denote $\bf b$ and $\bf f$ the base and fiber direction. Then we have 

\begin{claim}\label{fiber-deg-holder}
$|\nabla_{\bf f}^{g_i}\nabla^{j-1,g_i} u_i|_{g_i}\leq C\lambda_i^{k-j+\a}$ on $Q_{(z_i,y_i),i/3}$ for any $j\in \{1,...,k\}$.
\end{claim}
\begin{proof}
[Proof of Claim.] From \cite[claim 3, page 20]{HT18}, for each $x\in \pi_{\mathbb{C}^m}\left(B_{g_i}\left((z_i,y_i),i/3 \right)\right)$ and $|t|\leq i/3$ we have 
\begin{equation}
\begin{split}||\nabla_{\bf f}^{\lambda_i^2 g_Y}\nabla^{j-1,g_i}u_i(t)||_{\infty,\{z\}\times Y,\lambda_i^2g_Y}&\leq C\lambda_i^\a [\nabla^{k,g_i} u_i(t)]_{C^\a(\{z\}\times Y,\lambda_i^2 g_Y)}\\
&\leq C\lambda_i^{k-j+\a} [\nabla^{k,g_i} u_i(t)]_{\a,\a/2,Q_{(z_i,y_i),i/3}}\\
&\leq C\lambda_i^{k-j+\a}.
\end{split}
\end{equation}
Here we have used pullback of \eqref{1} under $F_i$. Since the estimate is uniform independent of $z$ and $t$, result follows. 
\end{proof}




With the equation \eqref{new-1}, \eqref{new-2}, \eqref{new-3} and  Claim \ref{fiber-deg-holder}, the proof of the Claim 3 in \cite[Proposition 3.9]{HT18} can now be carried over by replacing the operator $L^g$ with the heat operator $\Box_g$ and deriving contradiction  using  Proposition \ref{gradient-esti}. 
\end{proof}

Now we are ready to prove Theorem \ref{Sch-esti}.
\begin{proof}
[Proof of Theorem \ref{Sch-esti}.]
The proof is standard using Proposition \ref{sch-prep}. We include it for sake of completeness. Fix $\frac{1}{100}>\e>0$ and obtain $\delta_0(\e)$ and $C_1(\e)$ from Proposition \ref{sch-prep}. Let $\delta=\min\{\frac{1}{100},\delta_0\}$. Let $p,q\in B_{x,\rho}$ and connected by a unique $g_P$ minimal geodesic. Let $s,t\in (-\rho^2,0]$. If $d_{g_P}(p,q)+|s-t|^{1/2}\geq \delta (R-\rho)$, then 
$$\frac{|\nabla^{k,g_P} f(p,t)-\pt_{p,q}^{g_P} \nabla^{k,g_P}f(q,s)|_{g_P}}{\left(d_{g_P}(p,q )+|s-t|^{1/2}\right)^\a}\leq 2\left(\delta (R-\rho) \right)^{-\a}||\nabla^{k,g_P}f||_{\infty,Q_{p,\rho}}.$$

If $d_{g_P}(x,y)+|s-t|^{1/2}<  \delta (R-\rho)$, then apply Proposition \ref{sch-prep} with $Q_{p,R}$ replaced by $Q_{x,R-\rho}$ (here we may need to translate $t$ to 0 if necessary) to show that 
\begin{equation}
\begin{split}
\frac{|\nabla^{k,g_P} f(p,t)-\pt_{p,q}^{g_P} \nabla^{k,g_P}f(q,s)|_{g_P}}{\left(d_{g_P}(p,q)+|s-t|^{1/2}\right)^\a}&\leq \e[\nabla^{k,g_P} f]_{\a,\a/2,Q_{x,R}}+C[\nabla^{k-2,g_P}\Box f]_{\a,\a/2,Q_{x,R}}\\
&\quad +\sum_{j=0}^k C(R-\rho)^{-k+j-\a} ||\nabla^{j,g_P} f||_{\infty,Q_{x,\rho+\delta(R-\rho)}}\\
&\quad + \sum_{j=0}^{k-2} C(R-\rho)^{-k+2+j-\a}||\nabla^{j,g_P}\Box f||_{\infty,Q_{x,\rho+\delta(R-\rho)}}
\end{split}
\end{equation}

Hence in any case, we also have 
\begin{equation}
\begin{split}
[\nabla^{k,g_P} f]_{\a,\a/2,Q_{x,\rho}}&\leq \e[\nabla^{k,g_P} f]_{\a,\a,Q_{x,R}}+C[\nabla^{k-2,g_P}\Box f]_{\a,\a/2,Q_{x,R}}\\
&\quad +\sum_{j=0}^k C(R-\rho)^{-k+j-\a} ||\nabla^{j,g_P} f||_{\infty,Q_{x,\rho+\delta(R-\rho)}}\\
&\quad + \sum_{j=0}^{k-2} C(R-\rho)^{-k+2+j-\a}||\nabla^{j,g_P}\Box  f||_{\infty,Q_{x,\rho+\delta(R-\rho)}}
\end{split}
\end{equation}

By Lemma \ref{interpolation} with $\rho$ and $R$ replaced by $\rho+\delta(R-\rho)$ and $\rho+(\delta+\delta')(R-\rho)$, the last term can be replaced by the following.
\begin{equation}
\begin{split}
&\quad \sum_{j=0}^k C(R-\rho)^{-k+j-\a} ||\nabla^{j,g_P} f||_{\infty,Q_{x,\rho+\delta(R-\rho)}}\\
&\leq \tilde C_k\delta'^{\a}[\nabla^k f]_{\a,\a/2,Q_{x,R}}+C(R-\rho)^{-k-\a} ||f||_{\infty,Q_{x,R}}\\
&\leq \e [\nabla^k f]_{\a,\a/2,Q_{p,R}}+C(R-\rho)^{-k-\a}||f||_{\infty,Q_{x,R}}
\end{split}
\end{equation}
provided that $\delta'$ is sufficiently small. Similarly, 
\begin{equation}
\begin{split}
&\quad \sum_{j=0}^{k-2} C(R-\rho)^{-k+j+2-\a} ||\nabla^{j,g_P} \Box f||_{\infty,Q_{x,\rho+\delta(R-\rho)}}\\
&\leq C[\nabla^{k-2,g_P}\Box f]_{\a,\a/2,Q_{x,R}}+C'(R-\rho)^{-k-\a+2}||\Box f||_{\infty,Q_{x,R}}
\end{split}
\end{equation}

To summarize, we have shown that for all $0\leq \rho<R$,
\begin{equation}
\begin{split}
[\nabla^{k,g_P}f]_{\a,\a/2,Q_{x,\rho}}&\leq 2\e[\nabla^{k,g_P}f]_{\a,\a/2,Q_{x,R}}+C[\nabla^{k-2}\Box f]_{\a,\a/2,Q_{x,R}}\\
&\quad +C(R-\rho)^{-k-\a}||f||_{\infty,Q_{x,R}}\\
&\quad +C(R-\rho)^{-k-\a+2}||\Box f||_{\infty,Q_{x,R}}.
\end{split}
\end{equation}

Since $2\e<1$, this completes the proof by applying \cite[Lemma 3.4]{HT18}.
\end{proof}

\section{Local estimates of \KR flow}
\label{sect:LocalEstimates}
In this section, we will adapt the idea in \cite{HT18} and apply the cylindrical parabolic Schauder estimates to study the higher order regularity of \KR flow on $X$ with semi-ample canonical line bundle $K_X$. Let $(X,\omega_0)$ be a compact \K manifold and $\omega(t)$ be a long-time solution to the normalized \KR flow 
\begin{equation}\label{eq:NKRF} 
\partial_t \omega(t)=-\Ric(\omega(t))-\omega(t),\;\omega(0)=\omega_0.
\end{equation}

The semi-ample condition on $K_X$ induces a Calabi-Yau fibration structure $f : X^{m+n} \to \Sigma^m \subset \CP^N$ with possibly singular fibers. Denote the set of singular set of $\Sigma$ and critical value of $f$ by $S$. By \cite{ST12}, there exists a smooth \K metric $\omega_\Sigma$ on $\Sigma^n\setminus S$ satisfying the generalized K\"ahler-Einstein equation. Fix an open ball $B \subset\subset \Sigma \backslash S$, we may assume $B=B_1=B_{\mathbb{C}^m}(1)$ to be the Euclidean unit ball by rescaling and $\omega_\Sigma=\ddb v$ for some $v\in C^{\infty}(B)$. Throughout this section, we will consider the special case that for some $\e>0$, the regular fibers over $B_{1+\e}$ are biholomorphic to each other so that $f^{-1}(B_{1+\e})$ can be trivialized as a direct product $B_{1+\e}^m \times Y^n$ where $(Y,\omega_Y)$ is a closed \K manifold with $\Ric(\omega_Y)=0$. We will still denote $\omega_\Sigma$ and $v$ to be their pull-back on $B_{1+\e}\times Y$.

The main goal of this section is to prove the following local higher order regularity of $\omega(t)$. Under the above setting, it was already proved by the first author and Zhang \cite{FZ15} that there exists $C>1$ such that for all $t\in [0,+\infty)$, 
\begin{equation}\label{metric-equ}
C^{-1}\omega_P(t)\leq \omega(t)\leq C\omega_P(t),
\end{equation}
where $\omega_P(t)=\omega_{\mathbb{C}^m}+e^{-t}\omega_Y$.

\begin{thm}\label{KRF-local-regularity}
Under the above setting. Denote the reference $g_P(t)=g_{\mathbb{C}^m}+e^{-t}g_Y$, then for all $k\in \mathbb{N}$, there exist $C_k$'s such that for all $t\in [0,+\infty)$, 
$$\sup_{B\times Y}|\nabla^{k,g_P(t)}g(t)|_{g_P(t)}\leq C_k.$$
In particular, given any \K metric $g_X$ on $X$, compact set $\Omega$ away from singular fiber and $k\in \mathbb{N}$, there exist $C(k,\Omega,g_X)$ such that for all $t\in [0,+\infty)$, 
$$\sup_\Omega |\nabla^{k,g_X} g(t) |_{g_X}\leq C(k,\Omega,g_X).$$
\end{thm}

Now let us formulate the \KR flow setting. First we need to modify the set-up a bit in order apply the Schauder estimates. By \cite[Proposition 3.1]{Hein19}, there exists a biholomorphism $\Lambda$ of $B\times Y$ such that $\Lambda^*\omega_0=\omega_Y+\ddb u$ for some smooth function $u$. Note that \cite[Proposition 3.1]{Hein19} is stated with $B=\mathbb{C}^m$ but the proof also applies on Euclidean Ball $B$. Furthermore, $\Lambda$ is in form of $\Lambda(z,y)=(z,y+\sigma(z))$ for some holomorphic function $\sigma$ from $B$ to the space of $g_Y$-parallel $(1,0)$ vector fields on $Y$, we refer readers to \cite[(1.1)]{Hein19} for detailed exposition.

For a given normalized \KR flow $\omega(t)$, the pull-back of $\omega(t)$ by $\Lambda$ is also a solution to the normalized \KR flow on $B\times Y\times [0,+\infty)$. Let $\varphi$ be the solution to the following ordinary differential equation:
\begin{equation}\label{local-MA}
\left\{
\begin{array}{ll}
\displaystyle\frac{\partial \varphi}{\partial t}&=\displaystyle\log \frac{\left(\Lambda^*\omega(t)\right)^{m+n}}{\omega_P(t)^{m+n}}-\varphi-v;\\
\varphi(0)&=u
\end{array}
\right.
\end{equation}

By taking $\ddb$ on both sides and using the normalized K\'ahler-Ricci flow equation, one can easily show that 
\begin{equation}\label{local-MA-KRF-1}
\Lambda^* \omega(t)=e^{-t}\omega_Y+(1-e^{-t})\omega_\Sigma +\ddb \varphi.
\end{equation}
And hence \eqref{local-MA} is a local parabolic Monge-Amp\`ere equation. Note that we construct local solution to the parabolic Monge-Amp\`ere equation using the existence of \KR flow instead of standard PDE theory because $\omega_P(t)$ and $v$ are only locally defined. Note that $\Lambda^*\omega(t)$ is $d$-cohomologous on $B \times Y$ to $\omega_P(t)$ for each $t\geq 0$.

Theorem \ref{KRF-local-regularity} follows from the following proposition

\begin{prop}\label{KRF-local-regularity-pre}
Under the assumptions made in Theorem \ref{KRF-local-regularity}, for any $k\in \mathbb{N}$, there exists $C(B,k,m,n,Y,\omega_0)$ such that for all $t\in [0,+\infty)$,
$$\sup_{B\times Y}|\nabla^{k,g_P(t)}(\Lambda^*g(t))|_{g_P(t)}\leq C.$$
\end{prop}
Theorem \ref{KRF-local-regularity} follows quickly from Proposition \ref{KRF-local-regularity-pre}.
\begin{proof}
[Proof of Theorem \ref{KRF-local-regularity}]
The argument is almost identical to \cite[Corollary 1.3]{HT18}, we include the proof for reader's convenience. By Proposition \ref{KRF-local-regularity-pre}, for all $k\in \mathbb{N}$, there exist $C_k$'s such that for all $t\in [0,+\infty)$, 
$$\sup_{B\times Y}|\nabla^{k,g_P(t)}(\Lambda^*g(t))|_{g_P(t)}\leq C_k.$$

When $k=1$, due to \eqref{metric-equ}, it suffices to show that $|\nabla^{g_P(t)}-\nabla^{\Lambda^*g_P(t)}|_{g_P(t)}\leq C$ on $B\times Y$ for some $C$ independent of $t\rightarrow +\infty$. For each $t\in [0,+\infty)$, by rescaling and pulling back $g_P(t)$ and $\Lambda^*g_P(t)$ under the diffeomorphism given by $\Phi_t(z,y)=(e^{-t/2}z,y)$, it is equivalent to show $|\nabla^{g_P}-\nabla^{\hat g_P}|_{g_P}\leq Ce^{-t/2}$ for all $t>0$ where $g_P=g_{\mathbb{C}^m}+g_Y$ and $\hat g_P=\tilde\Phi_t^*g_P$ where $\tilde \Phi_t(z,y)=(z,y+\sigma(e^{-t/2}z))$. The estimates follows immediately as $|\nabla^{g_P}\hat g_P|_{g_P}\leq Ce^{-t}$ due to spatial stretching from pull back and the fact that $\sigma$ takes values in the $g_Y$ parallel vector fields on $Y$. The argument for $k>1$, note that
$$\nabla^{k}-\tilde \nabla^{k}=\nabla * (\nabla^{k-1}-\tilde\nabla^{k-1})+ (\nabla-\tilde\nabla)*\tilde \nabla^{k-1}.$$

By similar argument and induction, this is not difficult to prove that $\nabla^{k,g_P(t)}$ and $\nabla^{k,\Lambda^*g_P(t)}$ can be interchanged with a harmless error. This will complete the proof.
\end{proof}

It remains to prove Proposition \ref{KRF-local-regularity-pre}.
\begin{proof}
[Proof of Proposition \ref{KRF-local-regularity-pre}]
Let us fix some notations before we begin the proof. We will use $\omega_P(t)=\omega_{\mathbb{C}^m}+e^{-t}\omega_Y$ and $\omega_P=\omega_{\mathbb{C}^m}+\omega_Y$ to denote the product reference metric and product metric on $B\times Y$ respectively. Note that the connection induced by $\omega_P$ and $\omega_P(t)$ are identical due to the product structure. We will also denote $B_r$ to be ball of radius $r$ in $\mathbb{C}^m$.

For each $k\geq 0$, define a function $\mu_k(x,t)$ by 
\begin{equation}
\mu_k(x,t)=d_{g_P(t)}(x,\partial B\times Y)^k |\nabla^{k,g_P(t)} (\Lambda^*g(t))|_{g_P(t)}(x).
\end{equation}
It suffices to prove that each $\mu_k$ is uniformly bounded on $B\times Y\times [0,+\infty)$ which in turn implies uniform boundedness of $|\nabla^{k,g_P(t)}\Lambda^*g(t)|$ on $B_{1/2}\times Y\times [0,+\infty)$. And hence the main result will follow by appropriate rescaling or covering argument. When $k=0$, the boundedness of $\mu_0(x,t)$ has already been done using \eqref{metric-equ}. The main goal is to improve the regularity using \eqref{metric-equ}. 

Let $k\geq 1$. Suppose on the contrary that 
\begin{equation}\label{induction-assumption}
\sup_{B\times Y\times [0,+\infty)} \sum_{j=0}^{k-1}\mu_j(x,t) \leq C
\end{equation}
 for some $C>0$ and $\mu_k$ is not bounded uniformly as $t\rightarrow +\infty$. Then there exist sequences $x_i\in B\times Y$ and $t_i \rightarrow +\infty$ such that 
\begin{equation}
\label{mu-max}
\mu_k(x_i,t_i)=\max_{B\times Y\times [0,t_i]}\mu_k(x,t)\rightarrow +\infty.
\end{equation}

Define the rescaling factor by:
\[K_i := \abs{\nabla^{k,g_P(t_i)}\Lambda^*\left(g(t_i)\right)}_{g_P(t_i)}^{\frac{1}{k}}(x_i).\]
To see that $K_i \to +\infty$, we recall that $d_{g_P(t_i)}\left(x, \partial B \times Y)\right)$ is bounded above from $g_P(t)\leq g_P$, so we have
\[CK_i \geq d_{g_P(t_i)}\left(x_i, \partial B \times Y)\right)K_i = \mu_k(x_i,t_i)^{\frac{1}{k}}\]
which implies $K_i \to +\infty$. 

Define the biholomorphism $\Psi_i:B_{K_i}\times Y\rightarrow B\times Y$ by $$\Psi_i(z,y)=(K_i^{-1}z,y)$$ and denote $\hat x_i=\Psi_i^{-1}(x_i)$. Consider the parabolic rescaled metric
\begin{equation}
\label{rescaling-metric}
\left\{
\begin{array}{rl}
g_i(t)&=K_i^2\Psi_i^* \Lambda^*g(t_i+K_i^{-2}t),\\
g_{P,i}(t)&=K_i^2\Psi_i^*  g_P(t_i+K_i^{-2}t)=g_{\mathbb{C}^m}+K_i^2 e^{-t_i-K_i^{-2}t}g_Y
\end{array}
\right.
\end{equation}
 on $B_{K_i}\times Y\times (-K_i^2t_i,0]$ which satisfies   
 \begin{equation}\label{approximated-KRF}
 \begin{split}
\D{}{t}g_i(t) & = -\Ric(g_i(t)) - K_i^{-2}g_i(t).
 \end{split}
 \end{equation}

 By \eqref{metric-equ}, we still have 
 \begin{align}\label{metric-equ-1}
 C^{-1} g_{P,i}(t)\leq g_i(t)\leq Cg_{P,i}(t)
 \end{align}
 on $B\times Y\times [0,+\infty)$. On the other hand, for any $\hat{x} \in B_{K_i} \times Y$, one can easily verify that:
\begin{equation}
\label{eq:rescaled_mu}
\mu_k\big(\Psi_i(\hat{x}), t_i + K_i^{-2}t\big) = d_{g_{P,i}(t)}\big(\hat{x}, \partial B_{K_i} \times Y\big)^k \abs{\nabla^{k,g_{P,i}(t)}g_i(t)}_{g_{P,i}(t)}(\hat{x}).
\end{equation}

In particular, by \eqref{mu-max} 
\begin{equation}\label{completeness}
d_{g_{P,i}(0)}\big(\hat{x_i}, \partial B_{K_i} \times Y\big)=\mu_k(x_i,t_i)^{1/k}\rightarrow +\infty.
\end{equation}
where we have used 
\begin{equation}\label{in-parallel}
\abs{\nabla^{k,g_{P,i}(0)}g_i(0)}_{g_{P,i}(0)}(\hat{x_i})=1.
\end{equation}

Hence, the pointed limit with base point $\hat x_i$ will be complete. Moreover, since $\mu_k(\cdot,t_i)$ attains its maximum at $(x_i,t_i)$ and $g_{P,i}(t)$ is decreasing with respect to $t$, we can use triangle inequality to deduce that for all $\hat x\in B_{K_i}\times Y$, $t\in [-K_i^2t_i,0]$,
\begin{equation}\label{lma:derivative_rescaled}
\left\{
\begin{array}{rl}
|\nabla^{k,g_{P,i}(t)}g_i(t)|_{g_{P,i}(t)}(\hat x) &\leq \displaystyle \left(1-\frac{d_{g_{P,i}(0)}(\hat x,\hat x_i)}{d_{g_{P,i}(0)}(\hat x_i,\partial B_{K_i}\times Y)}\right)^{-k};\\
\displaystyle\sup_{B_{K_i/2}\times Y\times [-K_i^2t_i,0]}|\nabla^{j,g_{P,i}(t)}g_i(t)|_{g_{P,i}(t)}&\leq CK_i^{-j} \hfill \text{
for all } 0<j<k.
\end{array}
\right.
\end{equation}

The second inequality follows from \eqref{induction-assumption}. In particular, this gives the regularity of $g_i(t)$ with respect to the reference metric $g_{P,i}(t)$ which is possibly collapsing. Next, we consider the pointed limit of $(B_{K_i}\times Y,g_{P,i}(0),\hat x_i)$. By translation in $\mathbb{C}^m$, we may assume $\hat x_i=(0,y_i)\in \mathbb{C}^m\times Y$. Next we need to compare the rescaling \eqref{rescaling-metric} with the original collapsing speed. There are three possibilities, either $\delta_i \rightarrow+\infty$, $\delta_i\rightarrow \delta_0>0$ or $\delta_i\rightarrow 0$ where $\delta_i=K_ie^{-t_i/2}$. \\

{\bf CASE 1. $\delta_i\rightarrow +\infty$:} In this case, $g_{P,i}(t)$ does not converge as the $g_Y$ coefficient blows up. However, we can consider a local coordinate chart $(\Delta; y^1, \cdots, y^1)$ near the limit point
\[y_\infty := \lim_{i\to+\infty}\pi_Y(\hat{x}_i),\]
which exists after passing to a subsequence. We may assume $\Delta$ is the unit ball in $\C^n$, and denote by $\Delta_R$ the open ball with radius $R$ in $\C^n$. We further define a biholomorphism $\Phi_i : B_{K_i} \times \Delta_{K_i e^{-t_i/2}} \to B_{K_i} \times \Delta$ as
\[\Phi_i(z, y) = (z, \delta_i^{-1}y).\]
Express $g_Y$ in terms of local coordinates:
\[g_Y (y^1 \cdots, y^n)= 2\text{Re}\Big((g_Y)_{k\bar{l}}(y^1,\cdots,y^n)\;dy^k \otimes d\bar{y}^l\Big).\]
With \eqref{rescaling-metric}, this shows
\begin{align*}
& \big(\Phi_i^*{g}_{P,i}(t)\big)(z, y^1,\cdots,y^n)\\
& = g_{\C^m}(z) + e^{-K_i^{-2}t}\cdot 2\text{Re}\Big((g_Y)_{k\bar{l}}(\delta_i^{-1}(y^1,\cdots,y^n))\; dy^k \otimes d\bar{y}^l\Big)\\
& \to g_{\C^m}(z) + 2\text{Re}\Big((g_Y)_{k\bar{l}}(0,\cdots,0)\;dy^k \otimes d\bar{y}^l\Big) \;\; \text{ as $i \to +\infty$}.
\end{align*}
Note that we have used the fact that $\delta_i\rightarrow +\infty$.
In other words the pull-back metric $\Phi_i^*{g}_{P,i}(t)$ converges to the Euclidean metric, and so on compact subsets of $(\C^m \times Y) \times [0, \infty)$, we may assume $\Phi_i^*{g}_{P,i}(t)$ is uniformly equivalent to the Euclidean metric for large $i$.

Recall that from \eqref{metric-equ-1}, $\Phi_i^*{g}_{P,i}(t)$ and $\Phi_i^*g_i(t)$ are uniformly equivalent, and furthermore by pulling back \eqref{approximated-KRF}, $\Phi_i^*g_i(t)$ satisfies an ``approximated'' Ricci flow equation:
\[\D{}{t}\Phi_i^*g_i(t) = -\Ric(\Phi_i^*g_i(t)) - K_i^{-2}\Phi_i^*g_i(t).\]
Hence, by Theorem \ref{KRF-local-estimates}, we have for any $k \geq 1$ and any compact subset $\Omega \times [-\alpha, 0] \subset\subset \C^m \times \C^n \times (-\infty, 0]$, there exists a constant $C(k, \Omega)$ such that
\[\abs{\nabla^{k, \Phi_i^*g_{P,i}(t)}\big(\Phi_i^*g_i(t)\big)}_{\Phi_i^*g_{P,i}(t)} \leq C(k,\Omega) \;\; \text{ on } \Omega \times [-\alpha, 0]\]
for sufficiently large $i$ such that $(B_{K_i} \times Y) \times [-K_i^2t_i, 0] \supset \Omega \times [-\alpha,0]$.

With \eqref{completeness}, we then conclude that the following pointed manifold
\[\Big(B_{K_i} \times \Delta_{K_ie^{-t_i/2}} \times [-K_i^2t_i, 0], \; \Phi^*_i g_i(t), \;\Phi_i^{-1}\hat{x}_i\Big)\]
converges uniformly in $C^\infty$-Cheeger-Gromov sense to a complete limit space 
\[\Big(\C^m \times \C^n \times (-\infty, 0], \; g_\infty(t), \; \hat{x}_\infty\Big).\]

As $\Phi_i^*g_i(t)$ and $\Phi_i^*{g}_{P,i}(t)$ are all uniformly equivalent to the Euclidean metric (independent of both $t$ and $i$), the limit metric $g_\infty(t)$ is also uniformly equivalent to the Euclidean metric. Clearly, $g_\infty(t)$ satisfies 
\[\D{}{t}g_\infty(t) = -\Ric(g_\infty(t)), \; t \in (-\infty, 0].\]
By \cite{ST16}, the scalar curvature of the original solution $g(t)$ to \eqref{eq:NKRF} is always uniformly bounded (regardless of whether there are singular fibers and of the topological type of the fibers). This shows
\[\sup_{\left(B_{K_i} \times \Delta_{K_ie^{-t_i/2}}\right) \times (-K_i^2t_i, 0]}\abs{R\Big(\Phi_i^*K_i^2\Psi_i^*\Lambda^*g(t_i + K_i^{-2}t)\Big)} \leq \frac{C}{K_i^2}\]
and by letting $i \to +\infty$, the limit metric $g_{\infty}(t)$ is scalar flat and hence Ricci flat using
\[\heat R=|\Ric|^2.\]

As $g_\infty(0)$ is uniformly equivalent to the Euclidean metric on $\C^m \times \C^n$, it is parallel with respect to the Euclidean metric by Theorem \ref{Liouville-flat}. In particular, it shows for any $k \geq 1$, we have
\[\abs{\nabla^{k, g_{\C^{m+n}}}g_\infty(0)}_{g_{\C^{m+n}}} \equiv 0 \;\; \text{ on } \C^m \times \C^n.\]

On the other hand, by pulling back \eqref{in-parallel} under $\Phi_i$ and let $i\rightarrow +\infty$, we  have
\[\abs{\nabla^{k, g_{\C^{m+n}}}g_\infty(0)}_{g_{\C^{m+n}}}(\hat{x}_\infty) = 1\]
which is impossible.\\

{\bf CASE 2. $\delta_i\rightarrow \delta_0>0$:} We may assume $\delta_0=1$ by rescaling. Clearly, ${g}_{P,i}(t)$ converges to $g_{\C^m} + g_Y$ on $\C^m \times Y\times (-\infty,0]$ in $C^\infty_{\textup{loc}}$. With the Liouville's Theorem of K\"ahler Ricci-flat metrics on $\C^m \times Y$, Cases (1) and (2) are similar. The difference between Cases (1) and (2) is that in the later case we do not need to consider the biholomoprhism $\Phi_i$ to blow up $Y$ locally around $\hat{x}_i$.

Similar to Case (1), $g_i(t)$ satisfies an approximated Ricci flow equation
\[\partial_tg_i(t) = -\Ric(g_i(t)) - K_i^{-2}g_i(t).\]

By \eqref{metric-equ-1} and Theorem \ref{KRF-local-estimates}, one also has the local $C^k$-estimates (for any $k \geq 1$) for $g_i(t)$ with respect to $g_{\C^m} + g_Y$, and so the pointed space
\[\Big(B_{K_i} \times Y \times [-K_i^2t_i, 0], \; g_i(t), \; \hat{x}_i\Big)\]
converges in $C^\infty$-Cheeger-Gromov sense to a complete limit space
\[\Big(\C^m \times Y \times (-\infty, 0], \; g_\infty(t), \; \hat{x}_\infty\Big)\]
with $g_\infty(t)=g_\infty$ being a Ricci-flat K\"ahler metric for any $t \in (-\infty, 0]$. As $g_\infty(0)$ is uniformly equivalent to $g_{\C^m} + g_Y$ on $\C^m \times Y$, by Theorem \ref{Liouville-product} 
it is parallel with respect to $g_{\C^m} + g_Y$, but it contradicts to the fact that
\[\abs{\nabla^{k,{g}_{P,i}(0)}g_i(0)}_{{g}_{P,i}(0)}(\hat{x}_i) = 1\;\; \text{ for any $i$}.\]
Hence, this case is ruled out.\\

{\bf CASE 3. $\delta_i\rightarrow 0$:} 
Since $g_{P,i}(t)$ is product metric, we have $\nabla^{g_{P,i}(t)}=\nabla^{g_P}$. Using the fact that $g_{P,i}(t)\leq g_{P}$ for all $t<0$, \eqref{lma:derivative_rescaled}, \eqref{completeness} and \eqref{metric-equ-1} imply that for any $R>0$, there exists $C(R)$ such that for all $i\in \mathbb{N}$, $t\in [-K_i^2t_i,0]$, 
$$\sup_{B_R\times Y}\sum_{j=0}^k|\nabla^{j,g_P} g_i(t)|_{g_P}\leq C(R).$$
Hence by Ascoli-Aezel\`a Theorem, for each $t\in (-\infty,0]$, $g_i(t)$ converges in $C_{\textup{loc}}^{k-1,\a}$ to a $C_{\textup{loc}}^{k-1,\a}$ tensor $g_\infty(t)$ on $\mathbb{C}^m\times Y$ which is pullback of a $C^{k-1,\a}_{\textup{loc}}$ \K metric on $\mathbb{C}^m$ uniformly equivalent to $g_{\mathbb{C}^m}$ independent of $t$ using \eqref{metric-equ-1}. When $k=1$, $g_\infty(t)$ is \K in the sense that it is weakly closed. The main idea is to show that $g_\infty(t)$ satisfies $\omega_\infty(t)^m=c \omega_{\mathbb{C}^m}^m$ for some constant $c>0$ independent of $t$ and hence constant by Theorem \ref{Liouville-flat} which contradicts with \eqref{in-parallel}. First, we need a  slightly better regularity of $g_i(t)$ with respect to $g_{P,i}(t)$.\\

\noindent{\bf Claim 1.}
For all $\a\in (0,1)$, there exists $\e>0$ and $C>0$ such that for all $i\in \mathbb{N}$,
\begin{equation}
\label{im-regularity}[\nabla^{k+1, g_{P}}g_i(t)]_{\a,\a/2,B_\e\times Y\times (-\e^2,0],g_{P,i}(0)}\leq C.
\end{equation}
\begin{proof}
Let $\e>0$ be a constant to be determined and we will denote constants depending on $\e$ by $C_\e$. Let $\Phi_i:B_{\delta_i^{-1}}\times Y\rightarrow B\times Y$ be a biholomorphism given by  $\Phi_i(z,y)=(\delta_i z,y)$. Consider the parabolic rescaled metrics,
$$\eta_i(t)=\delta_i^{-2}\Phi_i^* \omega_i(\delta_i^{2}t) \quad\text{and}\quad  \eta_{P,i}(t)=\delta_i^{-2}\Phi_i^*  \omega_{P,i}(\delta_i^{2}t)$$
on $B_{\delta_i^{-1}}\times Y\times [-t_i e^{t_i},0]$. For convenience, we could write 
$$\eta_i(t)=e^{t_i}\Upsilon_i^* \Lambda^*\omega(t_i+e^{-t_i}t) \quad\text{and}\quad  \eta_{P,i}(t)=e^{t_i}\Upsilon_i^*\omega_P(t_i+e^{-t_i}t)$$
where $\Upsilon_i:B_{t_i/2}\times Y\rightarrow B\times Y$ is a biholomorphism given by $\Upsilon_i(z,y)=(e^{-t_i/2}z,y)$. Moreover, $\eta_i(t)$ is an approximated solution to the \KR flow:
\begin{equation}\label{almost-KRF}
\partial_t \eta_i(t)=-\Ric(\eta_i(t))-e^{-t_i} \eta_i(t).
\end{equation}

We may also write $\eta_{P,i}(t)=\omega_{\mathbb{C}^m}+\exp(-e^{-t_i}t)\omega_Y$ and $\eta_{P,i}(0)\equiv g_P=g_{\mathbb{C}^m}+g_Y$. Moreover,  \eqref{lma:derivative_rescaled} implies that for all $x\in B_{\delta_i^{-1}}\times Y $, $t\in [-t_ie^{t_i},0]$,
\begin{equation}\label{bootready-old}
\left\{
\begin{array}{ll}
C^{-1}\eta_{P,i}(t)\leq \eta_i(t)\leq C\eta_{P,i}(t);\\
|\nabla^{k,\eta_{P,i}(t)}\eta_i(t)|_{\eta_{P,i}(t)}\leq C\delta^k\\
|\nabla^{j,\eta_{P,i}(t)}\eta_i(t)|_{\eta_{P,i}(t)} \leq Ce^{-jt_i/2}\;\;\;\hfill \text{for all}\; j\in \{1,\cdots,k-1\}
\end{array}
\right.
\end{equation}
When $k=1$, the third inequality is an empty statement. Since $\eta_{P,i}(t)$ induces the same connection as $g_P$ and $\eta_{P,i}(t)$ is uniformly equivalent to $g_P$ on $B_{\delta_i^{-1}}\times Y\times [-\delta_i^{-2},0]$ independent of $i\rightarrow +\infty$, \eqref{bootready-old} can be replaced by
\begin{equation}\label{bootready}
\left\{
\begin{array}{ll}
C^{-1}\eta_{P,i}(t)\leq \eta_i(t)\leq C\eta_{P,i}(t);\\
|\nabla^{k,g_P}\eta_i(t)|_{g_P}\leq C\delta^k\\
|\nabla^{j,g_P}\eta_i(t)|_{g_P} \leq Ce^{-jt_i/2}\;\;\;\hfill \text{for all}\; j\in \{1,\cdots,k-1\}
\end{array}
\right.
\end{equation}
on $B_{\delta_i^{-1}}\times Y\times [-\delta_i^{-2},0]$. For notational convenience, we may assume $$B_{\delta_i^{-1}}\times Y\times [-\delta_i^{-2},0]=Q_{\delta_i^{-1}}$$ 
where $Q_{r}=B_{g_P}(\hat x,r)\times [-r^2,0]$ for some $\hat x\in \mathbb{C}^m\times Y$ (independent of $i$) since $\delta_i\rightarrow 0$. Now we are going to use the parabolic Monge-Ampere equation for $\eta_i(t)$ to improve the regularity. From \eqref{local-MA} and \eqref{local-MA-KRF-1}, for $s=t_i+e^{-t_i}t$ we have  \begin{equation}
\begin{split}
\eta_i^{m+n}(t)&=e^{(m+n)t_i}\Upsilon_i^*\Lambda^*\left( \omega(s)^{m+n}\right)\\
&=e^{(m+n)t_i}\Upsilon_i^*\left( e^{\dot\varphi(s)+\varphi(s)+v}\omega_{P}(s)^{m+n} \right)\\
&= e^{\Upsilon_i^*\dot\varphi(s)+\Upsilon_i^*\varphi(s)+\Upsilon_i^* v-ne^{-t_i}t} \tilde \omega_{P}^{m+n}.
 \end{split}
\end{equation}
where $$\tilde \omega_{P}(z,y)=\omega_{\mathbb{C}^m}(e^{-t_i/2}z)+\omega_Y(y).$$

To simplify the notation, we define 
$$ \phi_i(x,t)=e^{t_i}\varphi(\Upsilon_i(x),t_i+e^{-t_i}t)$$ 
so that 
\begin{equation}\label{local-eta-equ}
\left\{
\begin{array}{ll}
\displaystyle \log\frac{\eta_i^{m+n}(t)}{ \tilde \omega_{P}^{m+n}}=\dot\phi_i+e^{-t_i}\phi_i+\Upsilon_i^*v-ne^{-t_i}t\\\\
\eta_i(t)=\hat\eta_i(t)+\ddb \phi_i 
\end{array}
\right.
\end{equation}
where 
\begin{align}\label{local-eta}
\hat\eta_i(t)=e^{t_i}\Upsilon_i^* \left[ (1-e^{-s})\omega_{\Sigma}+e^{-s}\omega_{Y}\right].
\end{align}
Noted that 
$\hat\eta_i(t)\rightarrow \omega_P=\omega_{\mathbb{C}^m}+\omega_Y$
as $i\rightarrow +\infty$  in $C^\infty_{\textup{loc}}$ on $Q_{\delta_i^{-1}}$.

We will drop the index $i$ on $\delta_i$, $\phi_i$, $\eta_i(t)$ and $\hat \eta_i(t)$. Linearising the parabolic Monge-Amp\`ere equation yields
 \begin{equation}
\begin{split}
\int^1_0 \hat h_s^{k\bar l} \partial_k\partial_{\bar l}\phi \;ds&=\dot\phi+e^{-t_i}\phi+\Upsilon_i^*v -ne^{-t_i} t+ \log\frac{ \tilde \omega_P^{m+n}}{\hat\eta(t)^{m+n}}.
\end{split}
\end{equation}
where $\hat h_{i,s}(t)=s\eta_i(t)+(1-s)\hat\eta_i(t)$. If we define $\xi(x,t)=e^{2e^{-t_i}t}\phi(x,2t)$ for $t\in [-\delta_i^{-2}/2,0]$, then we have 
\begin{equation}\label{linearMA}
\begin{split}
\xi'-\Delta_{g_P}\xi=2\int^1_0\left(  h_s^{i\bar j}-g_P^{i\bar j}\right) \xi_{i\bar j} \;ds+2F
\end{split}
\end{equation}
where $ h_s(t)=\hat h_s(2t)$ and $F=ne^{-t_i+2e^{-t_i}t}+e^{2e^{-t_i}t}\log \frac{\hat\eta(2t)^{m+n}}{ \tilde\omega_{P}^{m+n}}-\Upsilon_i^*v$. Note that $\Delta=2\Delta_{\partial}$ by \K identity.


Noted that $\xi$ is unbounded in $L^\infty$, and hence the parabolic Schauder estimates do not apply directly on $\xi$. Instead, we apply on $\ddb \xi$. Taking $\ddb$ on both sides of \eqref{linearMA} and using the \K identity, we have 
\begin{equation}\label{linea-MA}
\begin{split}
\Box_{g_P} \ddb \xi=2\ddb\left[ \int^1_0\left( h_s^{i\bar j}-g_P^{i\bar j}\right)\xi_{i\bar j} \;ds\right]+2\ddb F
\end{split}
\end{equation}
where we have used the Hodge Laplacian. Since $e^{-t_i/2}\leq \delta_i$, we have 
\begin{equation}\label{source-term-skretching}
\left\{
\begin{array}{ll}
[\partial^k \ddb F]_{\a,\a/2,Q_{\delta^{-1}}}&\leq C\delta^{k+2+\a}\\
||\ddb F||_{\infty,Q_{\delta^{-1}}}&\leq C\delta^{2}
\end{array}
\right.
\end{equation}
which follows easily from spatial stretching of pull back since $v$ only depends on the base. By Theorem \ref{Sch-esti}, for any $0<\rho<R\leq (2\delta)^{-1}$,
\begin{equation}
\begin{split}
[\nabla^{k+1,g_P} \ddb \xi]_{\a,\a/2,Q_{\e \rho}}&\leq C\Big( [\nabla^{k-1,g_P}\Box_{g_P}\ddb \xi]_{\a,\a/2,Q_{\e \left(\rho+\frac{1}{2}(R-\rho)\right)}}\\
&\quad +(\e (R-\rho))^{-k-\a+1}||\Box\ddb \xi||_{\infty, Q_{\e \left(\rho+\frac{1}{2}(R-\rho)\right)}}\\
&\quad + (\e (R-\rho))^{-k-\a-1}||\ddb \xi||_{\infty, Q_{\e \left(\rho+\frac{1}{2}(R-\rho)\right)}}\Big)\\
&=\mathbf{I}+\mathbf{II}+\mathbf{III}.
\end{split}
\end{equation}

By \eqref{bootready}, 
$$\mathbf{III}\leq C_\e (R-\rho)^{-k-\a-1}.$$

For $\mathbf{I}$ and $\mathbf{II}$, we need to make strong use of the linearized Monge-Amp\`ere equation \eqref{linea-MA}. Since 
\begin{equation}
\begin{split}
&\quad \partial_k \partial_{\bar l} \left[ \int^1_0 (h_s^{i\bar j}-g_P^{i\bar j}) \xi_{i\bar j} ds\right]\\
&=\nabla^{g_P}_k \nabla^{g_P}_{\bar l} \left[ \int^1_0 (h_s^{i\bar j}-g_P^{i\bar j}) \xi_{i\bar j} ds\right]\\
&= \int^1_0  (h_s^{i\bar j}-g_P^{i\bar j})  \nabla^{g_P}_k \nabla^{g_P}_{\bar l} \xi_{i\bar j} +\nabla^{g_P}_k \nabla^{g_P}_{\bar l} h_s^{i\bar j} \xi_{i\bar j} +\nabla^{g_P}_{ k} h_s^{i\bar j}\nabla^{g_P}_{\bar l} \xi_{i\bar j} +\nabla^{g_P}_{\bar l} h_s^{i\bar j}\nabla^{g_P}_{k} \xi_{i\bar j} ds\\
&=\mathbf{A}+\mathbf{B}+\mathbf{C}+\mathbf{D}.
\end{split}
\end{equation}
\\

Since $\mathbf{C}$ and $\mathbf{D}$ are similar, we only consider $\mathbf{C}$. Since $\hat\eta$ is product metric, $\nabla^{g_P}=\nabla^{\hat\eta}$. Hence, $\nabla^{g_P}\hat\eta(t)\equiv 0$ and thus,
\begin{equation}
\begin{split}
\mathbf{C}&=\int^1_0 \nabla_k^{g_P} h_s^{i\bar j} \nabla^{g_P}_{\bar l}\xi_{i\bar j}ds=-\int^1_0 s h_s^{i\bar q} h_s^{p\bar j}  \nabla^{g_P}_{\bar l}\xi_{i\bar j}
\nabla^{g_P}_k \eta_{p\bar q}ds.
\end{split}
\end{equation}

Hence, \eqref{local-eta} and \eqref{bootready} will imply
\begin{equation}
||\mathbf{C}||_{\infty,Q_{\e R}}+||\mathbf{D}||_{\infty,Q_{\e R}}\leq C\delta^2.
\end{equation}

Using \eqref{bootready} with Lemma \ref{int-a},
\begin{equation}
\begin{split}
&\quad [\nabla^{k-1,g_P}\mathbf{C}]_{\a,\a/2,Q_{\e \left(\rho+\frac{1}{2}(R-\rho)\right)}}+[\nabla^{k-1,g_P}\mathbf{D}]_{\a,\a/2,Q_{\e \left(\rho+\frac{1}{2}(R-\rho)\right)}}\\
&\leq C\delta^{k+\a+1}+ C\delta [\nabla^{k,g_P} \ddb\xi]_{\a,\a/2,Q_{\e \left(\rho+\frac{1}{2}(R-\rho)\right)}}\\
&\leq C_\e\delta^{k+\a+1}+ C\e  [\nabla^{k+1,g_P} \ddb\xi]_{\a,\a/2,Q_{\e R}}.
\end{split}
\end{equation}

{We need more information of $\ddb \xi$ for $\mathbf{A}$ and $\mathbf{B}$.
\begin{subclaim}
For any $\e>0$, there is $N$ such that for all $i>N$, 
\begin{equation}
\label{decay}||\ddb \xi||_{\infty,Q_{\e\delta^{-1}}}\leq \e.
\end{equation}
\end{subclaim}
\begin{proof}
[Proof of subclaim]
From the uniform equivalence of metrics \eqref{bootready}, \eqref{almost-KRF} and Theorem \ref{KRF-local-estimates}, $\eta_i(t)$ is bounded locally uniformly on $Q_{\e\delta^{-1}}$ in any $C^k_{\textup{loc}}$ independent $i\rightarrow +\infty$. Hence, $\eta_i(t)$ converges in $C^\infty_{\textup{loc}}$ to a an ancient solution of unnormalized \KR flow $\eta_\infty(t)$ on $\mathbb{C}^m\times Y\times (-\infty,0]=Q_\infty$ which is parallel with respect to $\omega_P$ by \eqref{bootready}. Moreover, using the argument in {\bf CASE 2}, $\eta_\infty(t)\equiv \eta_\infty$ is a Ricci-flat metric and is $d$-cohomologous to $\omega_P$ by \eqref{local-eta-equ}, \eqref{local-eta} and \cite[Proposition 3.11]{HT18}. As pointed out in \cite[page 27]{HT18}, this implies $\eta_\infty$ differs from $\omega_P$ by a linear automorphism of $\mathbb{C}^m$. By pulling back the automorphism, we may assume $\eta_i(t)\rightarrow \omega_P$ as $i\rightarrow \infty$. In other word, $\ddb \xi_i(t)\rightarrow 0$ in $C^{\infty}_{\textup{loc}}(Q_{\infty})$ (after pulling back the automorphism).

From \eqref{bootready}, $||\nabla^{g_P} \ddb \xi||_{\infty,Q_{\delta^{-1}}}\leq C\delta$, and hence 
$$|\ddb \xi( x,t)|_{g_P}\leq |\ddb \xi(\tilde x,t)|_{g_P}+C\e$$
for all $( x,t)\in Q_{\e\delta^{-1}}$ where $\tilde x=(0,y)$ when $x=(z,y)$. It remains to prove that $|\ddb \xi(\tilde x,t)|\leq C\e$ for $t\in [-\e^2\delta^{-2},0]$ and sufficiently large $i$. Recall that 
\begin{equation}\label{flow-C^0}
\begin{split}
\ddb \xi(t)&=e^{e^{-t_i}2t} \ddb \phi_i(2t)\\
&=e^{\hat s} \Upsilon_i^* \ddb\varphi(\hat s).
\end{split}
\end{equation}
where $\hat s=t_i+e^{-t_i}2t\in [t_i-2\e^2 K_i^{-2},t_i]$. Let $i,j$ be the local coordinate on base $\mathbb{C}^m$, $\a,\b$ be the local coordinate on fiber $Y$. Since $ \eta_i(t)$ is positive definite and $\eta_{P,i}(t)$ is a product metric, it suffices to consider $\xi_{i\bar j}$ and $\xi_{\a\bar\b}$ by Cauchy-Schwarz inequality. By \cite[Theorem 1.2]{TWY}, for any compact set $\Omega$ away from the singular set, one have $\omega(t)\rightarrow \omega_\Sigma$ in $C^0_{\omega_0,\Omega}$ as $t\rightarrow +\infty$. This implies $\ddb  \varphi(t)\rightarrow 0$ in $C^0_{loc,\omega_0}$. In local coordinate, 
\begin{equation}
\begin{split}
 \ddb \xi_i(t)&=\sqrt{-1} \exp(e^{-t_i}2t)\varphi_{i\bar j}(\Upsilon_i(x),\hat s) dz^i\wedge d\bar z^{ j}\\
 &\quad +\sqrt{-1} \exp(t_i/2+e^{-t_i}2t) \varphi_{i\bar \b}(\Upsilon_i(x),\hat s) dz^i\wedge d\bar w^{ \b}\\
 &\quad +\sqrt{-1} \exp(t_i/2+e^{-t_i}2t) \varphi_{\a\bar j}(\Upsilon_i(x),\hat s) dw^\a\wedge d\bar z^{ j}\\
&\quad + \sqrt{-1} \exp(t_i+e^{-t_i}2t) \varphi_{\a\bar \b}(\Upsilon_i(x),\hat s) dw^\a\wedge d\bar w^{ \b}.
\end{split}
\end{equation}
where $\hat s=t_i+e^{-t_i}2t$. { Restricted to the base shows that for $i$ sufficiently large, 
\begin{equation}
\begin{split}
|\partial_i\partial_{\bar j} \xi(\tilde x,t)|_{g_{\mathbb{C}^m}}&\leq  \e
\end{split}
\end{equation}
since $\Upsilon_i(\tilde x)= (0,y)\in B\times Y$ and $\hat s\rightarrow +\infty$. For fiber direction, the proof is similar using our choice of $\omega_Y$ and \cite[Theorem 1.2]{TWY} that $e^t \omega(t)|_{X_0}\rightarrow \omega_{SRF,0}$ as $t\rightarrow +\infty$ on the fiber $f^{-1}(0)$. This proves the claim.}
\end{proof}

}

Using \eqref{decay}, \eqref{bootready},
\begin{equation*}
\begin{split}
\xi_{i\bar j}\nabla^{g_P}_k \nabla^{g_P}_{\bar l} h_s^{i\bar j} 
&=h_s^{-1} *h_s^{-1} *\nabla^{2,g_P} \ddb\xi *\ddb \xi\\
&\quad +h^{-1}_s* \nabla^{g_P}  \ddb\xi *\nabla^{g_P}  \ddb\xi *\ddb \xi,
\end{split}
\end{equation*}
Lemma \ref{interpolation} and Lemma \ref{int-a}, we have  
\begin{equation}
\left\{ 
\begin{array}{ll}||\mathbf{B}||_{\infty,Q_{\e \rho}} &\leq C_\e\delta^2;\\\\

[\nabla^{k-1,g_P}\mathbf{B}]_{\a,\a/2,Q_{\e \rho}}&\leq C\e [\nabla^{k,g_P}\ddb \xi]_{\a,\a/2,Q_{\e R}}+C_\e\delta^{k+\a+1}
\end{array}
\right.
\end{equation}

The term $\mathbf{A}$ can be estimated in a similar manner. Hence by combining with \eqref{source-term-skretching}, we get
\begin{equation}
\begin{split}
[\nabla^{k+1,g_P}\ddb\xi_i]_{\a,\a/2,Q_{x,\e\rho}}\leq  C\e [\nabla^{k+1,g_P}\ddb \xi]_{\a,\a/2,Q_{\e R}}+C_\e\delta^{k+\a+1}.
\end{split}
\end{equation}

Therefore by choosing $\e$ small, we can find $1>\e_0>0$ such that for all $0<\rho<R\leq (2\delta)^{-1}$ and $i$ sufficiently large, 
\begin{equation}
[\nabla^{k+1,g_P} \ddb \xi]_{\a,\a/2,Q_{\e \rho}}\leq \e_0 [\nabla^{k+1,g_P}\ddb \xi]_{\a,\a/2,Q_{\e R}}+C\delta^{k+\a+1}.
\end{equation}

By \cite[Lemma 3.4]{HT18}, $[\nabla^{k+1,g_P} \ddb \xi_i]_{\a,\a/2,Q_{\e (2\delta_i)^{-1}}}\leq C\delta_i^{k+\a+1}$ and hence the claim follows from pulling back to $g_i(t)$.
\end{proof}
\begin{rem}
We note that one should be able to establish a bound of any $C^l$ norm around the central fiber using interpolation argument.
\end{rem}

Thanks to {\bf Claim 1}, we have a better regularity on $g_i(t)$ around the central fiber at $t=0$. Therefore, $g_\infty(t)$ is at least $C^{k+\a,1+\a}_{\textup{loc}}$ around it. \\

\noindent{\bf Claim 2:}
\begin{equation}
\label{contradiction-point}
\abs{\nabla^{k,g_{\C^m}}g_{\infty}(0)}_{g_{\C^m}}(\hat{x}_\infty) = 1
\end{equation}
where $g_{\C^m}$ is the flat metric on $\C^m$.
\begin{proof}
The proof is identical to that of \cite[page 29, Claim 2]{HT18} by replacing the family of \K Ricci-flat metrics $\omega_t$ by \KR flow solution $\omega_i(0)$. See also \cite{MR2652468} for the origin of this argument.
\end{proof}

\noindent\textbf{Claim 3:} The $C^{k+\alpha}$ K\"ahler form $\omega_\infty(0)$ on $\C^m$ is parallel to the Euclidean metric. 
\begin{proof}
This follows from the estimates in \cite{TWY} and the argument in \cite{HT18} with some modifications. By \eqref{lma:derivative_rescaled}, we know that if $k>1$, then $$\abs{\nabla^{j,g_{P,i}(t)}g_i(t)}_{{g}_{P,i}(t)} \to 0$$ for all $j<k$ and so the claim is proved. Hence it suffices to consider $k = 1$.

Our goal is to show that $\omega_\infty=\lim_{i\rightarrow }\omega_i(0)$ satisfies 
$$\omega_\infty^m=c\omega_{\mathbb{C}^m}^m$$
for some constant $c$.

As shown at the beginning of proof of {\bf Claim 3}, $\omega_i(0)$ converges to $\omega_\infty$ in $C^\a_{loc,g_P}$ as a tensor on $\mathbb{C}^m\times Y$. By \eqref{metric-equ-1}, we can write $\omega_\infty=\omega_\Sigma|_{z=0}+\ddb \varphi_\infty$ for some $\varphi_\infty\in C^{2,\a}_{\textup{loc}}(\mathbb{C}^m)$. Here, with abuse of notations, we denote $\omega_\Sigma|_{z=0} = (g_{\Sigma})_{i\bar{j}}(0)\,dz^i \wedge d\bar{z}^j$ which is an Euclidean metric. Also denote $\varphi_\infty$ for the pull-back of $\varphi_\infty$ to $\mathbb{C}^m\times Y$. Recall from \eqref{local-MA-KRF-1} and \eqref{rescaling-metric} that
\begin{equation}\label{reMA-local-expan}
\omega_i(0)=(1-e^{-t_i})K_i^2 \Psi_i^* \omega_\Sigma+\delta_i^2 \omega_Y+K_i^2 \Psi_i^*\ddb\varphi(t_i),
\end{equation}
therefore we may assume $K_i^2 \Psi_i^*\ddb\varphi(t_i)\rightarrow \ddb \varphi_\infty$ as $K_ie^{-t_i/2}\rightarrow 0$. Denote $\varphi_i(x)=K_i^2\varphi(\Psi_i(x),t_i)$ so that $\ddb\varphi_i(0)\rightarrow \ddb\varphi_\infty$ in $C^\a_{\textup{loc}}$. Also denote $ \omega_{B,i}=(1-e^{-t_i})K_i^2 \Psi_i^* \omega_\Sigma$ so that $\omega_{B,i}\rightarrow \omega_\Sigma|_{z=0}$.

From \eqref{local-MA}, $\omega_i(0)$ satisfies
\begin{align}
\label{eq:MA-omega-i}
 \omega_i(0)^{m+n}& =\delta_i^{2n}e^{\Psi_i^*(\dot\varphi+\varphi+v)}\omega_P^{m+n}
\end{align}
where $\omega_P=\omega_{\mathbb{C}^m}+\omega_Y$. Here $\dot\varphi$ and $\varphi$ are evaluated at $t=t_i $. Denote by $\underline{\varphi}(z, t) : B \times [0,\infty) \to \R$ the average value of $\varphi(z,y,t) : (B \times Y) \times [0,\infty) \to \R$ over $Y$, i.e.
\[\underline{\varphi}(z,t) =\fint_Y \varphi(z,y,t)\,\omega_{Y}^n\]

Fix a test function $\eta \in C_c^\infty(\C^m)$ and assume $i$ is large enough so that $\text{supp}(\eta) \subset B_{K_i}$. Then, we have

\begin{equation}
\label{eq:volumeform_LHS}
\begin{split}
\delta^{-2n}_i\int_{\C^m \times Y} \eta\omega_i(0)^{m+n}
& = \int_{\C^m \times Y}e^{\Psi_i^*(\dot\varphi+\varphi+v)}\omega_P^{m+n}
\end{split}
\end{equation}

Note that $\omega_{B,i}(t)$ has only base components, so $(\omega_{B,i}(t))^j = 0$ for any $j > m$. so by expanding $\omega_i(t)^{m+n}$ we have
\begin{align*}
 &\quad \int_{\C^m \times Y} \eta e^{\Psi_i^*(\dot\varphi(t_i)+\varphi(t_i)+v)}\omega_P^{m+n}\\
 &=\delta_i^{-2n}\int_{\C^m \times Y} \eta \omega_i(0)^{m+n}\\
 &=\delta_i^{-2n}\int_{\C^m \times Y} \eta \left(\omega_{B,i}+\delta_i^2 \omega_Y+\ddb\varphi_i\right)^{m+n} \\
 &=\delta_i^{-2n}\int_{\C^m \times Y} \eta \left((\omega_{B,i}+\ddb \underline{\varphi_i})+(\delta_i^2 \omega_Y+\ddb(\varphi_i-\underline{\varphi_i}))\right)^{m+n}\\
 &=\delta_i^{-2n}\int_{\C^m \times Y} \eta\sum_{j=0}^{m+n} C^{m+n}_j (\omega_{B,i}+\ddb \underline{\varphi_i})^j \wedge (\delta_i^2 \omega_Y+\ddb(\varphi_i-\underline{\varphi_i}))^{m+n-j}
\end{align*}

Clearly, all the term with $j>m$ vanishes since $\omega_{B,i}+\ddb\underline{\varphi_i}$ are from base only. For those terms with $j<m$, we now claim that they all converge to $0$ as $i\rightarrow \infty$. By expanding $(\delta_i\omega_Y+\ddb (\varphi_i-\underline{\varphi_i}))^{m+n-j}$, it suffices to consider the following integral where $0\leq k< m+n-j$.
\begin{equation}
\label{eq:integral_j<m}
\begin{split}
&\quad \delta_i^{-2n} \int_{\C^m \times Y} \eta (\omega_{B,i}+\ddb \underline{\varphi_i})^j \wedge (\delta_i^2\omega_Y)^k\wedge (\ddb (\varphi_i-\underline{\varphi_i}))^{m+n-j-k}\\
&=\quad \delta_i^{-2n}   \int_{\C^m \times Y} (\varphi_i-\underline{\varphi_i})\ddb \eta\wedge (\omega_{B,i}+\ddb \underline{\varphi_i})^j \\
&\quad \quad\quad\quad\quad\quad\quad \wedge  (\delta_i^2\omega_Y)^k\wedge (\ddb (\varphi_i-\underline{\varphi_i}))^{m+n-j-k-1}.
\end{split}
\end{equation}
 
By \eqref{metric-equ-1}, we know that for any $z\in \Omega$, $\left|\ddb (\varphi_i-\underline{\varphi_i})|_{z\times Y} \right|=O(\delta_i^2)$ and hence $|\varphi_i-\underline{\varphi_i}|=O(\delta_i^2)$ by applying Yau's $C^0$-estimates \cite{Yau1978} on the each fiber $\{z\} \times Y$ to the metric $\frac{1}{\delta_i^2}\omega_i(0)\big|_{\{z\}\times Y}$ which, according to \eqref{metric-equ-1} and \eqref{eq:MA-omega-i}, has uniformly bounded volume form. When restricted on the fiber $\{z\} \times Y$, $\frac{1}{\delta_i^2}\omega_i(0)$ takes the form
\[\omega_Y + \sqrt{-1}\ddbar\big(\delta_i^{-2}(\varphi_i - \underline{\varphi_i})\big)\]
where $\omega_Y$ is a fixed metric, so one can apply Yau's estimate directly.

Moreover, \eqref{metric-equ-1} implies the base components of $\ddb (\varphi_i-\underline{\varphi_i})$ is uniformly bounded, so by Cauchy-Schwarz inequality, the mixed base-fiber components of $\ddb (\varphi_i-\underline{\varphi_i})$ are of $O(\delta_i)$.

As $\eta$ is independent of $Y$, by counting the contribution to the fiber direction, we claim that when $j < m$ the integral \eqref{eq:integral_j<m} converges to $0$ as $i\rightarrow +\infty$. It is because each term in
\[\ddb \eta\wedge (\omega_{B,i}+\ddb \underline{\varphi_i})^j \wedge  (\delta_i^2\omega_Y)^k\]
has $j+1$ many $dz$'s and $d\bar{z}$'s, and $k$ many $dy$'s and $d\bar{y}$'s. By wedging with $\ddb (\varphi_i-\underline{\varphi_i}))^{m+n-j-k-1}$, only terms with $n-k$ many $dy$'s and $d\bar{y}$'s would not be annihilated. To summarize, the integral in \eqref{eq:integral_j<m} is of order:
\[\delta_i^{-2n} O(\delta_i^2) O(1)O(1)^j  O(\delta_i^2)^k \cdot O(\delta_i^2)^{n-k} = O(\delta_i^2).\]
It proves our claim that the integral in \eqref{eq:integral_j<m} converges to $0$ as $i \to +\infty$.

It remains to consider $j=m$, modulo the constants which is 
\begin{equation}
\begin{split}
\delta_i^{-2n} \int_{\C^m \times Y} \eta (\omega_{B,i}+\ddb \underline{\varphi_i})^m \wedge (\delta_i^2\omega_Y)^k\wedge (\ddb (\varphi_i-\underline{\varphi_i}))^{n-k}.
\end{split}
\end{equation}

If $n>k$, then using integration by parts and counting the fiber component contributions in a similar manner as the above would show that they are zero. Hence it remains to consider 
\begin{equation}
\begin{split}
&\quad  \int_{\C^m \times Y} \eta (\omega_{B,i}+\ddb \underline{\varphi_i})^m \wedge \omega_Y^n.
\end{split}
\end{equation}

For any compactly supported smooth function $\eta$ on $\mathbb{C}^m$,
\begin{align*}
& \int_{\C^m \times Y}\eta e^{\Psi_i^*(\dot\varphi(t_i)+\varphi(t_i)+v)}\omega_P^{m+n}\\
& =C^{m+n}_m\left(\int_{\mathbb{C}^m }\eta (\omega_{B,i}+\ddb \underline{\varphi_i})^m \right) \left(\int_Y \omega_Y^n \right) + o(1) 
\end{align*}
as $i \to +\infty$. Observe that we also have 
$$\omega_{B,i}+\ddb\underline{\varphi_i}\rightarrow \omega_{\Sigma}|_{z=0}+\ddb\varphi_\infty=\omega_\infty.$$

On the other hand, by \cite[Lemma 3.1]{TWY}, after performing pull-back, we have $\dot \varphi+\varphi\rightarrow 0$ as $t\rightarrow +\infty$ on $B\times Y$. Hence, by Fubini's Theorem for all $\eta\in C_c^\infty(\C^m)$,
$$\int_{\C^m}\eta \omega_P^{m}=c_0\int_{\mathbb{C}^m }\eta \omega_\infty^m $$
for some constant $c_0>0$. This completes the proof.
\end{proof}

Thanks to \eqref{metric-equ-1}, by regularity of elliptic Monge-Amp\`ere equation, $\omega_\infty$ is smooth and Ricci flat. Then contradictions arise from Theorem \ref{Liouville-flat} and {\bf Claim 2}. This completes the proof of Proposition \ref{KRF-local-regularity}
\end{proof}

The uniform boundedness of Ricci curvature away from the singular fibre is immediate.
\begin{proof}
[Proof of Corollary \ref{cor:RicciBound}]
For two \K metrics $g$ and $h$, we have 
\begin{equation}
\begin{split}
|\Ric_g|_h &\leq |\Ric_g-\Ric_h |_h +|\Ric_h|_h\\
&=\left|\ddb \log \frac{\det h}{\det g}\right|_h +|\Ric_h|_h\\
&\leq |\nabla^{2,h} F|_h +|\Ric_h|_h
\end{split}
\end{equation}
where $F=\log  \frac{\det h}{\det g}$. Note that  
\begin{equation}
\nabla^{2,h} F= g^{-1}*\nabla^{2,h} g+ g^{-1} * g^{-1} * \nabla^h g *\nabla^{h} g
\end{equation}
the Ricci curvature bound of $g(t)$ follows immediately from above inequality and Theorem \ref{KRF-local-regularity} by substituting $g=g(t)$ and $h=g_P(t)$ since $g(t)$ is uniformly equivalent to $g_P(t)$ independent of $t>0$.
\end{proof}

\bibliographystyle{amsplain}

\end{document}